%% file: main-fiber.tex
\documentclass[a4paper,10pt,reqno,oneside]{amsart}
\usepackage {graphics}
\usepackage{mathbbol}
\usepackage{amssymb}
\usepackage{color}
\usepackage{tabulary}
\usepackage[ruled,section]{algorithm}
\usepackage{algorithmic}
\usepackage[round,comma]{natbib}
\usepackage{booktabs}
\usepackage{multicol,ragged2e}
\usepackage{ifpdf}

\newtheorem{Theorem}{Theorem}[section]

\newtheorem{Lemma}[Theorem]{Lemma}

\newtheorem{Definition}[Theorem]{Definition}
\newtheorem{Remark}[Theorem]{Remark}

\newtheorem{Example}[Theorem]{Example}

\newtheorem{Notation}[Theorem]{Notation}

\newcommand{\red}{\sqsubseteq}

\newcommand{\Orthant}{\mathbb O}

\newcommand{\koplus}[1]{\overset{#1}{\oplus}}

\DeclareMathOperator{\newt}{Newt}
\DeclareMathOperator{\supp}{supp}
\DeclareMathOperator{\conv}{conv}

\newcommand {\st} {{\rm s.t. }\hspace{0.5mm}}

\newcommand{\R}{{\mathbb R}}
\newcommand{\Q}{{\mathbb Q}}
\newcommand{\Z}{{\mathbb Z}}

\newcommand\fieldk{\mathsf k}

\usepackage[breaklinks,colorlinks,citecolor=blue]{hyperref}

\begin{document}

\title{Computation of Atomic Fibers of $\Z$-Linear Maps}
\author{Elke Eisenschmidt}
\address{Elke Eisenschmidt, Otto-von-Guericke-Universit\"at Magdeburg,
  Department of  Mathematics, Institute for Mathematical Optimization
  (IMO), Universit\"atsplatz 2, 
  39106 Magdeburg, Germany}
\email{eisensch@imo.math.uni-magdeburg.de}

\author{Raymond Hemmecke}
\address{Raymond Hemmecke, Otto-von-Guericke-Universit\"at Magdeburg,
  Department of Mathematics, Institute for Mathematical Optimization
  (IMO), Universit\"atsplatz 2, 
  39106 Magdeburg, Germany} 
\email{hemmecke@imo.math.uni-magdeburg.de}
\thanks{The second author wishes to thank Jes\'us De Loera
for some useful suggestions on improving the presentation of an earlier
version of this paper.}

\author{Matthias K\"oppe}
\address{Matthias K\"oppe, Otto-von-Guericke-Universit\"at Magdeburg,
  Department of Mathematics, Institute for Mathematical Optimization
  (IMO), Universit\"atsplatz 2, 
  39106 Magdeburg, Germany} 
\email{mkoeppe@imo.math.uni-magdeburg.de}

\date{$\relax$Revision: 1.32 $ - \ $Date: 2007/04/20 12:55:58 $ $}

\begin{abstract}
For given matrix $A\in\Z^{d\times n}$, the set
$P_{b}=\{z:Az=b,z\in\Z^n_+\}$ describes the preimage or fiber
of $b\in\Z^d$ under the $\Z$-linear map
$f_A:\Z^n_+\rightarrow\Z^d$, $x\mapsto Ax$. The fiber $P_{b}$
is called atomic, if $P_{b}=P_{b_1}+P_{b_2}$ implies
$b=b_1$ or $b=b_2$. In this paper we present a novel algorithm to
compute such atomic fibers. An algorithmic solution to appearing
subproblems, computational
examples and applications are included as well.
\end{abstract}

\maketitle

%\paragraph{Key Words.}
%\paragraph{AMS subject classifications.}

\section{Introduction}
Decomposition of rational polyhedra is at the heart of several
interesting applications. However, there are different definitions
of decomposability depending on the application. These definitions
mainly differ in the treatment of the (integer) points of the
polyhedron.

The simplest notion is that of linear decomposition of polyhedra.
Two polyhedra $P,Q\subseteq\R^n$ are called {\em homothetic} if
$P=\lambda Q+t$ for some $\lambda>0$ and $t\in\R^n$. Here, a
polyhedron $P$ is called {\em indecomposable}, if any
decomposition $P=Q_1+Q_2$ implies that both $Q_1$ and $Q_2$ are
homothetic to $P$. It can be shown that there are only finitely
many indecomposable rational polyhedra that are not homothetic to
each other. For further details on this type of decomposition we
refer the reader for example to
\cite{Gruenbaum:67,henk-koeppe-weismantel-2000:miptestset,Kannan+Lovasz+Scarf:93,
McMullen:73,Meyer:74,Smith:87}.

Let us now come to a bit more restrictive decomposition. Here we
consider only polyhedra of the form $\{x\in\R^n:Ax\leq b\}$ for
a given matrix $A\in\Z^{d\times n}$ and varying $b\in\Z^d$. To
emphasize that we only consider integer right-hand sides, we say
that a polyhedron $P$ is {\em integrally indecomposable}, if any
decomposition $P=Q_1+Q_2$ (into polyhedra with integer right-hand
sides) implies that both $Q_1$ and $Q_2$ are homothetic to $P$.
This decomposition is more restrictive than the linear
decomposition, since only such polyhedra $Q_1$ and $Q_2$ are
allowed that have an {\em integer} right-hand side.
\cite{henk-koeppe-weismantel-2000:miptestset} showed finiteness of
the system of integrally indecomposable polytopes. This result implies
important applications: TDI-ness of each member of a family of systems
$Ax\leq b$, $b\in\Z^d$, can be concluded from TDI-ness of the
integrally indecomposable systems. Furthermore the finiteness of the
system of integrally indecomposable polytopes enables us to compute a
finite representation of a test set for a mixed-integer linear optimization
problem.

Another important application of integral decomposition of
polyhedra is that of factorizing a multivariate polynomial, see for
example \cite{Salem+Gao+Lauder:04} 
and the references therein. Here, one considers only polyhedra of the
form $\{x\in\R^n:Ax\leq b\}$ for given matrix $A\in\Z^{d\times n}$ and
varying $b\in\Z^d$, where each polyhedron is integer, that is, where
each polyhedron has only integer vertices. Note that the notion of
integral decomposability is restricted to integral polyhedra in this
application whereas the definition of
\cite{henk-koeppe-weismantel-2000:miptestset} is valid for arbitrary
rational polyhedra with integral right-hand side. The reason for this
restriction is the simple observation that the so-called Newton polytope
$\newt(f):=\conv\{\alpha_i \in \supp(f)\}$ associated to a
polynomial $f = \sum\nolimits_{i \in I}{a_i x^{\alpha_i}}$ with
$\supp(f) = \{\alpha_i \colon a_i \neq 0\}$ is integer by
definition. Moreover, the relation $f=gh$ among three polynomials $f$,
$g$, and $h$ implies 
$\newt(f)=\newt(g)+\newt(h)$, a theorem due to Ostrowski. 

A direct generalization of the above notion of integral decomposition of
integral polyhedra was introduced by \cite{Adams+Hosten+Loustaunau+Miller:99}.
They considered polytopes
\begin{displaymath}
  \tilde P_{b} := \conv \{ z : Az=b, z\in\Z^n_+\}
\end{displaymath}
called the \emph{fibers} of $b$ under the linear map $f_A\colon
\Z^n_+\rightarrow\Z^d$, $x\mapsto Ax$.  A fiber $\tilde P_b$ is called
\emph{atomic} if $\tilde P_{b}=\tilde P_{b_1}+\tilde P_{b_2}$ implies $b=b_1$
or $b=b_2$.  Note that $\tilde P_{b}=\tilde P_{b_1}+\tilde P_{b_2}$ means that
\emph{every vertex} of $\tilde P_{b}$ is the sum of a \emph{vertex} of $\tilde
P_{b_1}$ and a \emph{vertex} of $\tilde P_{b_2}$ (and vice versa).
Atomic fibers were used by \cite{Adams+Hosten+Loustaunau+Miller:99} to
construct strong SAGBI bases for subalgebras of polynomial rings. 
They proved that the family of atomic fibers is finite and
also gave an algorithm to compute atomic fibers via certain standard
pairs. Via this algorithm, \cite{Adams+Hosten+Loustaunau+Miller:99} computed
the atomic fibers of the 
twisted cubic, see Example~\ref{Example: Atomic Fibers
of Twisted Cubic}.\smallbreak

In this paper, we consider a slight variation of the notion of atomic
fibers that was introduced by \citet{Maclagan}.  Instead of considering convex hulls~$\tilde P_{b}$ of the preimages
\[
P_{b}:=\{z:Az=b,z\in\Z^n_+\},
\]
of the map $f_A\colon \Z^n_+\rightarrow\Z^d$, $x\mapsto Ax$, we consider the
preimages $P_b$ themselves.  In \citeauthor{Maclagan}'s more general
terminology, the sets $P_b$ are called $((0), A)$-fibers; we shall simply call
them \emph{fibers} in the remainder of this paper.
We call a fiber $P_{b}$ \emph{indecomposable} or \emph{atomic}, if
$P_{b}=P_{b_1}+P_{b_2}$ implies $b=b_1$ or $b=b_2$.
Note that $P_{b}=P_{b_1}+P_{b_2}$ means that {\em
every lattice point} of $P_{b}$ is the sum of a \emph{lattice
point} of $P_{b_1}$ and a \emph{lattice point} of
$P_{b_2}$ (and vice versa). This is indeed a very strong
condition, but again it was shown that there are only finitely many
(nonempty) atomic fibers for a given matrix $A$ \citep{Maclagan}. 
Note that atomic fibers are not only
minimal (with respect to decomposability) within the given family,
but also generate every fiber $P_{b}$ in this family as a
Minkowski sum $P_{b}=\sum\nolimits_{i=1}^{k} \alpha_iP_{b_i}$,
$\alpha_i\in\Z_+$, where $\alpha_i P_{b_i}$ stands for iterated
Minkowski-addition of $P_{b_i}$ with itself.
Atomic fibers (of this kind) were used in 
the computation of minimal vanishing sums of roots of unity
\citep{Steinberger:04}.  

Recently, the computation of atomic fibers also
appeared as a subproblem in the capacitated design of
telecommunication networks for a given communication demand
under survivability conditions \citep{EKL:2006}. 
In this application, the right-hand side vectors~$b$ are taken from a
sublattice or a submonoid of~$\Z^d$. Of course, restricting the set of
``feasible'' right-hand sides changes the notion of decomposability, thus the
set of atomic fibers is changed. 

Another related notion is that of \emph{extended} atomic fibers.
We call the set 
\[
Q_{b}:=\{z:Az=b,z\in\Z^n\}
\]
an \emph{extended fiber} of the linear map of~$A$.
We call it \emph{atomic}, if $(Q_{b}\cap\Orthant_j)=
(Q_{b_1}\cap\Orthant_j)+(Q_{b_2}\cap\Orthant_j)$ holds for
all the $2^n$ orthants $\Orthant_j$ of $\R^n$, then $b=b_1$ or
$b=b_2$. Here, as well, it can be shown that there are only finitely
many (nonempty) extended atomic fibers for a given matrix. Also
this very strong notion of decomposability has an application: the
set ${\mathcal H}_\infty$ constructed in \cite{Hemmecke:SIP2} for use
in two-stage stochastic integer programming is in fact the set of
extended atomic fibers of the family of extended fibers
\[
\{(x,y):x=b,Tx+Wy=0,x\in\Z^m,y\in\Z^n\}
\]
where $T$ and $W$ are kept fixed and where $b$ varies.

\subsection*{Outline}
In this paper, we are mainly concerned about designing efficient algorithms for
computing atomic and extended fibers.
The outline of the paper is as follows. In section~\ref{Section: Atomic
  fibers} we first define a hierarchy of \emph{partially
  extended fibers} that interpolate between fibers and extended fibers.
This hierarchy not only generalizes the notions of fibers and extended fibers,
but also plays a significant r\^ole in our algorithms.
Motivated by our application in survivable network design, we define 
decomposability with respect to a given finitely generated monoid of feasible
right-hand side 
vectors.  We prove that, in this more general situation as well, there are only
finitely many atomic fibers.  
We also present an algorithmic way to
decompose a fiber into a Minkowski sum of indecomposable fibers. 

In section \ref{Section: Computation of Atomic Fibers} we present a first
algorithm to compute the atomic extended fibers of a given matrix, following
the pattern of a completion procedure.  We present the algorithm in a
simplified setting where the right-hand side vectors are restricted to a
sublattice (rather than a submonoid) of~$\Z^d$. 
By restricting the atomic extended fibers to the positive orthants and
performing a simple reduction step, 
the atomic fibers (or partially extended fibers) of a matrix can be easily
obtained.  However, this method is not a very efficient one for computing atomic
fibers. 

Therefore, we present a more efficient way to compute
atomic fibers via a project-and-lift approach in section
\ref{Section: Preliminaries of the project-and-lift algorithm} and
\ref{Section: The k-th step of the project-and-lift algorithm}.  We present
the method in the general setting where a finitely generated monoid of
right-hand sides is given by its generators.

Both our algorithms enable us to compute not only the atomic fibers
$P_b$ but also the atomic fibers $\tilde{P}_b$ according to the
definition in \cite{Adams+Hosten+Loustaunau+Miller:99}. This will be
shown at the end of section \ref{Section: Computation of Atomic Fibers}.

Finally, in section \ref{Section: First computational results}, we present first
computational results of the project-and-lift algorithm.

\section{(Partially Extended) Atomic Fibers}
\label{Section: Atomic fibers}
Let us now start our treatment with a formal definition of partially
extended fibers.

\begin{Definition}
\label{partial fibers}
Let $A \in \Z^{d \times n}$ be a matrix, $b \in \Z^d$ and $0 \leq k
\leq n$.
\begin{enumerate}
\item[(i)]\addvspace{-8pt}
The set 
$$ Q_{b}^{(k)} := \{z \, : \, Az = b, \, z \in \Z_+^k \times
\Z^{n-k}\}$$
is called an \emph{partially extended fiber of order $k$} of the matrix
$A$. The set $Q_b := Q_b^{(0)}$ is called an \emph{extended fiber}, and $P_b :=
Q_b^{(n)}$ is called a \emph{fiber} of the matrix $A$.
\item[(ii)]
Let $0 \leq l \leq n$.
For $u,v \in \R^n$ we say that $u \red_l v$ if $u^{(i)}v^{(i)} \geq
0$ and $|u^{(i)}| \leq |v^{(i)}|$ for all components $i =1, \ldots,
l$. We will abbreviate $\red_n$ by $\red$. For $U,V,W \subseteq
\R^n$ we say that 
\begin{displaymath}
U = V \koplus{(l)} W
\end{displaymath}
and call $U$ the \emph{$l$-restricted Minkowski sum} of $V$ and $W$,
if for all $u \in U$ there exist $v \in V$, $w \in W$ with $v,w
\red_l u$ and $u = v + w$. Note that $V \mathop{\oplus}^{(0)} W$ is just
the ordinary Minkowski sum $V + W$. We will abbreviate
$\mathop{\oplus}^{(n)}$ by $\oplus$.
\item[(iii)]
For $0 \leq m \leq n$ we will denote by $\pi_m \colon \R^n \rightarrow
\R^m$ with $(x_1, 
\ldots, x_n) \mapsto (x_1, \ldots, x_m)$ the projection onto the first
$m$ components.
\end{enumerate}
\end{Definition}

Now we will go on defining \emph{atomic} partially extended fibers w.r.t.~a
certain monoid $M \subseteq \Z^d$.  To accompany the hierarchy of partially
extended fibers, we define a hierarchy of notions of decomposition that
interpolates between ordinary Minkowski sums and orthant-wise Minkowski sums. 

\begin{Definition}
\label{defining atomic fibers}
Let $A \in \Z^{d \times n}$ be a matrix, $b \in \Z^d$ and $0 \leq k,l
\leq n$. Additionally, let $M \subseteq \Z^d$ be a monoid. 
\begin{enumerate}
\item[(i)]\addvspace{-8pt} We call $Q_b^{(k)}$ \emph{atomic
w.r.t.~$\mathop{\oplus}^{(l)}$ and $M$} if
there is no decomposition
\begin{displaymath}
Q_b^{(k)} = Q_{b_1}^{(k)} \mathop{\oplus}^{(l)} Q_{b_2}^{(k)}
\end{displaymath}
with $b_1,b_2 \in M$ and
$\pi_l(Q_{b_1}^{(k)}),\pi_l(Q_{b_2}^{(k)}) \neq \pi_l(Q_0^{(k)})$.
By $E^{(k)}_l(A,M)$ we denote the set of partially extended fibers of
order $k$ which are atomic w.r.t.~$\mathop{\oplus}^{(l)}$ and $M$.
\item[(ii)]
We denote by $E^{(k)}(A,M)$
the set $E_n^{(k)}(A,M)$ and call it the set of \emph{partially extended atomic
fibers w.r.t.~the monoid $M$}. We denote by $F(A,M)$ the set
$E^{(n)}(A,M)$ and call it the set of \emph{atomic fibers w.r.t.~$M$}.
\end{enumerate}
\end{Definition}

Note that Definition \ref{defining atomic fibers} also applies to the
special case where the monoid $M$ is a lattice. We will see later on
that it is 
much easier to compute the atomic (partially extended) fibers of a
matrix w.r.t.~a lattice instead of an arbitrary monoid.

As our first step, we prove a generalization of the finiteness result
for the family of atomic fibers.
\begin{Lemma}
\label{reduction and monoid-atomic}
Let $0 \leq k \leq n$ be fixed.
There are only finitely many partially extended fibers
$Q^{(k)}_b$ which are atomic w.r.t.~a finitely generated monoid $M$.
\end{Lemma} 
The proof of this lemma is based on the following nice theorem.

\begin{Theorem} [\citealp{Maclagan}] \label{Theorem: Maclagan} 
  Let $\fieldk$ be a field. Let ${\mathcal I}$ be an infinite family of
  monomial ideals  
  in a polynomial ring $\fieldk[x_1,\ldots,x_n]$. Then there must exist
  ideals $I,J\in {\mathcal I}$ with $I\subseteq J$.
\end{Theorem}

To apply this theorem in our situation of partially extended fibers
which are atomic w.r.t.~a certain finitely generated monoid $M$, we
introduce the following definition. 

\begin{Definition}
\label{reduction}
Let $A \in \Z^{d \times n}$ and $M = \langle m_1, \ldots, m_t \rangle
\subseteq \Z^d$ a finitely generated monoid.  Let $0 \leq k \leq n$ be fixed.
\begin{enumerate}
\item[(i)]\addvspace{-8pt}
Let $\alpha, \bar{\alpha} \in \Z_+^t$ with 
\begin{displaymath}
b :=\sum\nolimits_{i=1}^{t}{\alpha_i m_i} \quad \text{ and } \quad \bar{b} :=
\sum\nolimits_{i=1}^{t}{\bar{\alpha}_i m_i}.
\end{displaymath}
We say that
$(\bar{\alpha}, Q^{(k)}_{\bar{b}})$ \emph{reduces} $(\alpha, Q^{(k)}_{b})$ 
and denote 
$$(\bar{\alpha}, Q^{(k)}_{\bar{b}}) \trianglelefteq (\alpha, Q^{(k)}_{b})$$
if $\bar{\alpha} \red \alpha$ and $Q^{(k)}_{b} = Q_{\bar{b}}^{(k)}
\oplus Q^{(k)}_{b - \bar{b}}$. In particular: $b -\bar{b} \in M$.
\item[(ii)]
We call a pair $(\alpha, Q^{(k)}_{b})$ \emph{irreducible
w.r.t.~$\trianglelefteq$} if there is no pair
$(\bar{\alpha},Q^{(k)}_{\bar{b}})$ different from $(\alpha,
Q^{(k)}_{b})$ and $(0,Q^{(k)}_{0})$ with 
$$(\bar{\alpha},Q^{(k)}_{\bar{b}}) \trianglelefteq
(\alpha,Q^{(k)}_{b}).$$
\end{enumerate}
\end{Definition}

\begin{Lemma}
\label{finitely many irreducible pairs}
Let $0 \leq k \leq n$ be fixed. Let ${\mathcal A} =
\{(\alpha^1,Q_{b_1}^{(k)}),(\alpha^2,Q_{b_2}^{(k)}), \ldots \}$ be a
set of pairs.
\begin{enumerate}
\item[(i)]\addvspace{-8pt}
Let $(\alpha^i,Q_{b_i}^{(k)}) \ntrianglelefteq (\alpha^j,
Q_{b_j}^{(k)})$ for all $(\alpha^i,Q_{b_i}^{(k)}),
(\alpha^j,Q_{b_j}^{(k)}) \in {\mathcal A}$ with $i < j$. 
Then ${\mathcal A}$ is finite.
\item[(ii)] There are only finitely many pairs $(\alpha,Q_b^{(k)})$
  which are irreducible w.r.t.~$\trianglelefteq$.
\end{enumerate}
\end{Lemma}
\begin{proof}
(i):
We associate with a pair $(\alpha^j,Q^{(k)}_{b_j})$ the monomial
  ideal
 \begin{equation*}
\begin{split}
I_{\alpha} &= \langle x^{(z_1, \ldots, z_k, z_{k+1}^+,z_{k+1}^-, \ldots, z_{n}^+, z_{n}^-, \alpha^j_1,
  \ldots, \alpha^j_t)} \colon Az = 
  \sum\limits_{i=1}^{n}{\alpha^j_i m_i} (\;= b_j), \; z \in \Z_+^{k}
  \times \Z^{n-k} \rangle\\
&\subseteq \Q[x_1, \ldots, x_{2n+t}],
\end{split}
\end{equation*}
where $z_i^+ = \mathrm{max }\{0,z_i\}$ and $z_i^- = \mathrm{max
}\{0,-z_i\}$. Then $(\alpha^j,Q_{b_j}^{(k)}) \ntrianglelefteq
(\alpha^l,Q_{b_l}^{(k)})$ if $I_{\alpha^j}$ is not contained in
$I_{\alpha^l}$. Consider the set ${\mathcal I} =
\{I_{\alpha^1},I_{\alpha^2}, \ldots \}$ of ideals associated to the
elements in the set ${\mathcal A}$. The set ${\mathcal I}$ then is an
antichain of ideals and is thus finite according to \ref{Theorem: Maclagan} (see
\cite{Maclagan}). The finiteness of ${\mathcal A}$ follows from the
finiteness of ${\mathcal I}$. \newline
(ii): 
A pair $(\alpha,Q^{(k)}_{b})$ is irreducible
w.r.t.~$\trianglelefteq$ if and only if $(\alpha,Q^{(k)}_{b})
\ntrianglelefteq (\bar{\alpha},Q_{\bar{b}}^{(k)})$  for any $\bar{\alpha} \neq
\alpha$. Let ${\mathcal A} = \{(\alpha^1, Q_{b_1}^{(k)}),
  (\alpha^2,Q_{b_2}^{(k)}), \ldots \}$ be the set of pairs which are
  irreducible w.r.t.~$\trianglelefteq$. Part (i) then yields that
  ${\mathcal A}$ is finite.  \end{proof}

We are now ready to prove Lemma~\ref{reduction and monoid-atomic}.

\begin{proof}[Proof of Lemma~\ref{reduction and monoid-atomic}]
It is sufficient to show: for every $Q^{(k)}_b$ atomic w.r.t.~$M$ there
exists $\alpha \in \Z_+^t$ with $b = \sum\nolimits_{i=1}^{t}{\alpha_im_i}$ such
that $(\alpha,Q^{(k)}_b)$ is irreducible w.r.t.~$\trianglelefteq$. Then
there is an injective mapping from the set of atomic extended fibers
$Q^{(k)}_b$ into the set of irreducible pairs $(\alpha,Q^{(k)}_b)$ and thus there
are only finitely many extended atomic fibers w.r.t.~$M$.

Let $b$ be fixed with $Q^{(k)}_b$ an extended atomic fiber w.r.t.~$M$. 
Let $\alpha \in \Z_+^t$ with $b =
\sum\nolimits_{i=1}^{t}{\alpha_im_i}$ be minimal w.r.t.~$\red$, i.e.,
there is no $\Z_+^t \owns \bar{\alpha} \neq \alpha$ with $\bar{\alpha}
\red \alpha$ and $b = \sum\nolimits_{i=1}^{t}{\bar{\alpha}_im_i}$. We
claim that the pair $(\alpha,Q^{(k)}_b)$ is irreducible
w.r.t.~$\trianglelefteq$. Suppose not. Then there is $(\bar{\alpha},
Q^{(k)}_{\bar b}) \trianglelefteq (\alpha,Q^{(k)}_b)$, i.e., $\bar{\alpha} \red
\alpha$ and $Q^{(k)}_b = Q^{(k)}_{\bar{b}} \oplus Q^{(k)}_{b - \bar{b}}$ implying $b -
\bar{b} \in M$. As $Q^{(k)}_b$ is an
  extended atomic fiber we may w.l.o.g. assume that $\bar{b} = b$ and
  $b - \bar{b} = 0$. Therefore $b =
  \sum\nolimits_{i=1}^{t}{\bar{\alpha}_im_i}$ and as $\alpha$ is
  minimally chosen w.r.t.~$\red$ we have $\bar{\alpha} = \alpha$.
This proves our claim.
\end{proof}

\begin{Example} \label{Example: Atomic Fibers of Twisted Cubic}
In \cite{Adams+Hosten+Loustaunau+Miller:99}, it was shown how
atomic fibers could be used to construct strong SAGBI bases for
monomial subalgebra over principal ideal domains. As an example,
they computed the atomic fibers of the matrix
$A=\left(\begin{smallmatrix} 3 & 2 & 1 & 0\\ 0 & 1 & 2 & 3\\
\end{smallmatrix}\right)$ by hand via an approach different from
the one we present below.

In the table below, we list the right-hand sides and all (finitely
many) elements in these $18$ atomic fibers. 

\[
\begin{array}{ll}
(0,3) & \{(0,0,0,1)\}\\
(1,2) & \{(0,0,1,0)\}\\
(2,1) & \{(0,1,0,0)\}\\
(3,0) & \{(1,0,0,0)\}\\
(2,4) & \{(0,1,0,1),(0,0,2,0)\}\\
(3,3) & \{(1,0,0,1),(0,1,1,0)\}\\
(4,2) & \{(0,2,0,0),(1,0,1,0)\}\\
(3,6) & \{(1,0,0,2),(0,1,1,1),(0,0,3,0)\}\\
(4,5) & \{(0,2,0,1),(0,1,2,0),(1,0,1,1)\}\\
(5,4) & \{(1,1,0,1),(0,2,1,0),(1,0,2,0)\}\\
(6,3) & \{(2,0,0,1),(1,1,1,0),(0,3,0,0)\}\\
(4,8) & \{(0,2,0,2),(1,0,1,2),(0,1,2,1),(0,0,4,0)\}\\
(6,6) & \{(2,0,0,2),(0,3,0,1),(1,1,1,1),(1,0,3,0),(0,2,2,0)\}\\
(8,4) & \{(2,1,0,1),(0,4,0,0),(1,2,1,0),(2,0,2,0)\}\\
(6,9) & \{(2,0,0,3),(0,3,0,2),(1,1,1,2),(1,0,3,1),(0,2,2,1),(0,1,4,0)\}\\
(9,6) & \{(3,0,0,2),(1,3,0,1),(2,1,1,1),(2,0,3,0),(1,2,2,0),(0,4,1,0)\}\\
(6,12) & \{(2,0,0,4),(0,3,0,3),(1,1,1,3),(1,0,3,2),(0,2,2,2),(0,1,4,1),(0,0,6,0)\}\\
(12,6) & \{(4,0,0,2),(2,3,0,1),(3,1,1,1),(3,0,3,0),(2,2,2,0),(0,6,0,0),(1,4,1,0)\}\\
\end{array}
\]
Thus, for example, the fiber given by the right-hand side $(8,7)$
is not atomic, since it can be decomposed into atomic fibers as
\[
P_{\left(\begin{smallmatrix}8\\7\\\end{smallmatrix}\right)}=
P_{\left(\begin{smallmatrix}2\\4\\\end{smallmatrix}\right)}\oplus
P_{\left(\begin{smallmatrix}6\\3\\\end{smallmatrix}\right)}.
\]
This can be quickly verified by looking at the elements in these
fibers:
\begin{eqnarray*}
& &
\{(2,1,0,2),(2,0,2,1),(1,1,3,0),(1,2,1,1),(0,4,0,1),(0,3,2,0)\}\\
& = &
\{(0,1,0,1),(0,0,2,0)\}\oplus\{(2,0,0,1),(1,1,1,0),(0,3,0,0)\}.
\end{eqnarray*}
Indeed, we have
\begin{eqnarray*}
(2,1,0,2) & = & (0,1,0,1)+(2,0,0,1), \\
(2,0,2,1) & = & (0,0,2,0)+(2,0,0,1), \\
(1,1,3,0) & = & (0,0,2,0)+(1,1,1,0), \\
(1,2,1,1) & = & (0,1,0,1)+(1,1,1,0), \\
(0,4,0,1) & = & (0,1,0,1)+(0,3,0,0), \\
(0,3,2,0) & = & (0,0,2,0)+(0,3,0,0).
\end{eqnarray*}
\qed
\end{Example}
In Example \ref{Example: Atomic Fibers of Twisted Cubic} above, it
was easy to verify whether a given fiber is a summand in the
decomposition of another fiber by simply checking the finitely
many elements in the fiber for a decomposition. If the fibers are
not bounded, however, this would not give a finite procedure. The
following lemma tells us how to solve this problem via the
(finitely many!) $\red$-minimal elements in the given fibers.

\begin{Definition}
\label{minimal elements}
Let $A \in \Z^{d \times n}$ and $b \in \Z^d$. Let $0 \leq k \leq l
\leq n$.
\begin{enumerate}
\item[(i)]\addvspace{-8pt}
An element $v \in Q_b^{(k)}$ is called \emph{minimal w.r.t.~$\red_l$}
if there is no $w \in Q_b^{(k)}$ with $v \neq w$ and $w \red_l v$.
\item[(ii)]
We define $z, \tilde{z} \in
Q_b^{(k)}$ to be equivalent if and only if $\pi_l(z) =
\pi_l(\tilde{z})$.
\end{enumerate}
For $l < n$ there are infinitely many
$\sqsubseteq_{l}$-minimal elements in general. Therefore we have to
restrict ourselves to \emph{representatives of equivalence classes} of
$\sqsubseteq_{l}$-minimal elements.  Let $R_{b,l}^{(k)}$
denote a set of representatives of the equivalence classes of
the $\sqsubseteq_{l}$-minimal elements in $Q_{b}^{(k)}$.
Let these representatives be chosen arbitrarily but fixed.
\end{Definition}

\begin{Remark}
\label{finiteness of minimal elements}
Let $A \in \Z^{d \times n}$, $b \in \Z^d$ and $0 \leq k,l \leq
n$. Then the set of representatives of $\red_l$-minimal elements in
$Q^{(k)}_b$, $R_{b,l}^{(k)}$, is finite by the Lemma of Gordan--Dickson
(see for example \cite{Cox-Little-O'Shea:1992}).
\end{Remark}

\begin{Lemma}
\label{minimal elements suffice}
Let $0 \leq k \leq l \leq n$ and let $Q^{(k)}_{b_1}\neq\emptyset$,
$Q^{(k)}_{b_2}\neq\emptyset$. 
Then $Q^{(k)}_{b_1+b_2}=Q^{(k)}_{b_1}\mathop{\oplus}^{(l)} Q^{(k)}_{b_2}$ if and only
if for every $\red_l$-minimal vector $v\in R_{b_1+b_2,l}^{(k)}$ there is
a vector $w\in Q^{(k)}_{b_1}$ with $w\red_l v$.
\end{Lemma}
\begin{proof}  Let $v \in Q_{b_1+b_2}^{(k)}$. Then there is $\bar{v}\in
R^{(k)}_{b_1+b_2,l}$ with
$\bar{v}\red_l v$. Thus, by the assumption in the lemma, there is
some $\bar{w}\in Q^{(k)}_{b_1}$ such that $\bar{w}\red_l\bar{v} \red_l
v$. 
As $k \leq l$ we have $\bar{v} - \bar{w} \in \Z_+^k \times \Z^{n-k}$
and thus $\bar{v}-\bar{w}\in Q^{(k)}_{b_2}$ with
$\bar{v}-\bar{w}\red_l\bar{v} \red_l v$.

We now claim that $v=(\bar{w}+v-\bar{v})+(\bar{v}-\bar{w})$ with
$\bar{w}+v-\bar{v}\in Q^{(k)}_{b_1}$, $\bar{v}-\bar{w}\in
Q^{(k)}_{b_2}$, $\bar{w}+v-\bar{v}\red_l v$, and $\bar{v}-\bar{w}\red_l
v$, is a desired representation of $v$. The first two relations
are trivial, if we keep in mind that $Av=A\bar{v}=b$,
$A\bar{w}=b_1$, $b=b_1+b_2$ and $k \leq l$. We get the other two relations as
follows:
\begin{itemize}
\item [(a)] $\bar{w}+v-\bar{v}\red_l \bar{v}+v-\bar{v}=v$, since by
construction $\pi_l(\bar{w})$ and $\pi_l(v-\bar{v})$ lie in the same orthant,
and
\item [(b)] $\bar{v}-\bar{w}\red_l \bar{v}\red_l v$, since
$\bar{w}\red_l\bar{v}$.
\end{itemize}

Thus, we have constructed for arbitrary $v\in Q^{(k)}_{b_1+b_2}$ a
valid representation of $v$ as a sum of two elements from
$Q^{(k)}_{b_1}$ and $Q^{(k)}_{b_2}$ whose projection onto the first
$l$ components lie in the same orthant as the projection of
$v$ onto its first $l$ components. This concludes the proof. \end{proof}

Using this lemma repeatedly, we are now able to find, for a given
right-hand side $b \in M$, a decomposition
$Q_{b}^{(k)}=\bigoplus_{i=1}^{s}\alpha_iQ_{b_i}^{(k)}$,
$\alpha_i\in\Z_+$, that is, we can find a decomposition of a partially
extended
fiber into a sum of partially extended fibers which are atomic
w.r.t.~the monoid $M$.

\begin{algorithm} 
\caption{Algorithm to decompose extended fibers into sums of extended atomic
fibers}
\label{Algorithm to decompose fibers into sums of atomic fibers}
\begin{algorithmic}[1]
\REQUIRE{$A$, right-hand sides $\{b_1,\ldots,b_s\}$ of the set of
  extended atomic fibers $E^{(k)}(A,M)$}
\ENSURE{ $\alpha_1,\ldots,\alpha_s$ such that
$Q^{(k)}_{b}=\bigoplus\limits_{i=1}^{s} \alpha_iQ^{(k)}_{b_i}$}

\STATE $\alpha_1:=\ldots:=\alpha_s:=0$
\FOR{$ i = 1$ to $s$} %{$ k $ }
\WHILE{$Q_{b}^{(k)}=Q_{b_i}^{(k)}\oplus Q_{b-b_i}^{(k)}$ and $b-b_i \in M$} 
\STATE $b:=b-b_i$
\STATE $\alpha_i:=\alpha_i+1$
\ENDWHILE
\ENDFOR
\STATE \textbf{return: } $\alpha_1,\ldots,\alpha_s$.
\end{algorithmic}
\end{algorithm}

It remains to state an algorithm that computes the finitely many
$\red$-minimal elements in $Q^{(k)}_{b}$ for fixed $k$.
We will do this in the following paragraphs.

%As one subproblem in the computation of (extended) atomic fibers,
We
have to find for some $l\in\{1,\ldots,n\}$ and some
$k\in\{1,\ldots,l\}$ all $\red_l$-minimal elements in (projections
of) fibers of the form
\[
\pi_l(Q_b^{(k)}) = \{(x,y)\in\Z_+^{k}\times\Z^{(l-k)}:\exists\; z\in\Z^{(n-l)} \text{
with } A(x,y,z)=b \}.
\]
If $b=0$, then $0$ is the only $\red_l$-minimal element. If not, we
reduce this problem to the problem of finding a Hilbert basis of a
cone. It is not hard to show that all $\red_l$-minimal elements
$(x,y,z)$ correspond to the elements $(x,y^+,y^-,z,1)$ in a Hilbert
basis of the cone
\[
\{((x,y^+,y^-,z,u)\in\Z^{n+(l-k)+1}:A(x,y^+-y^-,z)-bu=0,
x,y^+,y^-,u\geq 0\}.
\]
In general, this is not a pointed rational polyhedral cone (and thus
need not have a unique inclusion-minimal Hilbert basis), since there
can be linear relations among the (free) variables $z$. However,
projected onto the space of the variables $x,y^+,y^-,u$, the
nonnegativity constraints lead to a pointed rational polyhedral cone
that possesses a unique inclusion-minimal Hilbert basis. Such a
minimal Hilbert basis can be computed for example with {\tt 4ti2}
(see \citealp{4ti2}).

Note that the splitting of $y$ into $y^+$ and $y^-$ is only used for
exposition here. In practice, one can directly use $y$ when
computing the $\red_l$-minimal elements, see \citet[Section
2.6]{Ray:Habil} for more details.

\section{Computation of (Extended) Atomic Fibers}
\label{Section: Computation of Atomic Fibers}

In the following we show how to compute the finitely many
(extended) atomic fibers of a matrix $A \in \Z^{m \times n}$ w.r.t.~a
lattice $\Lambda$. In this section we will present a simple algorithm;
we will give a more complex and much more efficient algorithm in the following
sections. Both algorithms use the algorithmic pattern of a completion procedure. 

We will
denote the columns of matrix $A$ by $A_1, \ldots, A_n \in
\Z^m$. Note that the function $\textrm{normal form}(s,G)$ in Algorithm
\ref{Algorithm to compute atomic fibers} stems from Algorithm
\ref{Normal form algorithm}. 

\begin{algorithm}
\caption{Algorithm to compute extended atomic fibers}
\label{Algorithm to compute atomic fibers}
\begin{algorithmic}[1]
\REQUIRE{ $F:=\{\pm b_1, \ldots, \pm b_s\}$ with $\langle b_1, \ldots,
  b_s \rangle = \Lambda \cap A\Z^n$}
\ENSURE{a set $G$, such that $\{Q_{b}:b\in G\}$ contains
all extended fibers of $A$ which are atomic w.r.t.~$\Lambda$}
\STATE $G:=F$

\STATE 
$C:=\bigcup\limits_{f,g\in G}\{f+g\}$ \hspace{2.5cm} (forming S-vectors)
\WHILE{$C\neq \emptyset $} 

\STATE $s:=$ an element in $C$

\STATE $C:=C\setminus\{s\}$

\STATE $f:=\textrm{normal form }(s,G)$

\IF{ $f\neq 0$} 

\STATE$G:=G\cup \{f\}$

\STATE $C:=C\cup\bigcup\nolimits_{g\in G} \{f+g\}$ \hspace{1.5cm}
(adding S-vectors)
\ENDIF
\ENDWHILE
\STATE $G:=G\cup\{0\}$
\STATE \textbf{return: } $G$.
\end{algorithmic}
\end{algorithm}

\begin{algorithm}
\caption{Normal form algorithm}
\label{Normal form algorithm}
\begin{algorithmic}[1]
\REQUIRE{ $s$, $G$}

\ENSURE{a normal form of $s$ with respect to $G$}

\WHILE{there is some $g\in G$ such that
$Q_{s}=Q_{g}\oplus Q_{s-g}$} 
\STATE $s:=s-g$
\ENDWHILE
\STATE \textbf{return: } $s$
\end{algorithmic}
\end{algorithm}

\begin{Lemma} \label{Lemma: Algorithm to compute extended atomic fibers
terminates and is correct} Algorithm \ref{Algorithm to compute
atomic fibers} terminates and computes a set $G$ such that
$\{Q^I_{A,b}:b\in G\}$ contains all atomic fibers of $A$.
\end{Lemma}

\begin{proof}  Associate with $b \in \Lambda$ the monomial ideal
$I_{A,b}:=\langle x^{(z^+,z^-)}:Az=b,z\in\Z^n\rangle\subseteq
\Q[x_1,\ldots,x_{2n}]$, where $(z^+)^j = \mathrm{max }(0,z^j)$ and
$(z^-)^j = \mathrm{max }(0,-z^j)$ for all components $j=1, \ldots,
n$. Algorithm \ref{Algorithm to compute atomic 
fibers} generates a sequence $\{f_1,f_2,\ldots\}$ in $G\setminus
F$ such that $Q_{f_j}\neq Q_{f_i}\oplus Q_{f_j-f_i}$
whenever $i<j$. Thus, the corresponding sequence
$\{I_{A,f_1},I_{A,f_2},\ldots\}$ of monomial ideals satisfies
$I_{A,f_j}\nsubseteq I_{A,f_i}$ whenever $i<j$. We conclude, by
Theorem \ref{Theorem: Maclagan} given by \cite{Maclagan}, that this sequence of monomial
ideals must be finite and thus, Algorithm \ref{Algorithm to compute atomic
fibers} must terminate.

It remains to prove correctness. For this, let $G$ denote the set
that is returned by Algorithm \ref{Algorithm to compute atomic
fibers}. Moreover, let $Q_{\bar{b}}$ be an extended atomic
fiber of $A$ with $\bar{b}\neq 0$. We will show that $\bar{b}\in
G$.

Since $F\setminus \{0\} \subseteq G\setminus\{0\}$, we know that
$Q_{\bar{b}}=\sum Q_{b_j}$ for finitely many (not
necessarily distinct) $b_j\in G\setminus\{0\}$. This implies in
particular, that every $z\in Q_{\bar{b}}$ can be written as a
sum $z=\sum v_j$ with $v_j\in Q_{b_j}$. We will show that we
can find vectors $b_j\in G$ such that every $z\in Q_{\bar{b}}$ can be
written as a sum $z=\sum v_j$ with $v_j\in Q_{b_j}$ and $v_j\red
z$. This implies 
$Q_{\bar{b}}=\bigoplus Q_{b_j}$. Since $Q_{\bar{b}}$
is atomic, and thus indecomposable, this representation must be
trivial, that is, it has to be $Q_{\bar{b}}=Q_{\bar{b}}$,
and therefore we conclude $\bar{b}\in G$.

With Lemma \ref{minimal elements suffice} it is sufficient to
consider the $\red$-minimal elements in $Q_{\bar{b}}$,
$R_{\bar{b},n}^{(0)}=\{z_1,\ldots,z_k\}$, to decide if it decomposes
w.r.t.~$\oplus$.
From all representations $Q_{\bar{b}}=\sum_{j\in J}
Q_{b_j}$ with $b_j\in G\setminus\{0\}$ choose a representation and
elements $v_{i,j} \in Q_{b_j}$ with $z_i=\sum_{j\in J} v_{i,j}$ 
$i=1,\ldots,k$, such that the sum
\begin{equation}\label{sum}
\sum_{i=1}^k\sum_{j\in J}\|v_{i,j}\|_1
\end{equation}
is minimal. By the triangle inequality we have that
\begin{equation}\label{triangle inequality}
\sum_{i=1}^k\sum_{j\in J}\|v_{i,j}\|_1\geq\sum_{i=1}^k \|z_i\|_1.
\end{equation}
Herein, equality holds if and only if all $v_{i,j}$ have the same
sign pattern as $z_i$, $i=1,\ldots,k$, that is, if and only if we
have $v_{i,j}\red z_i$ for all $i$ and all $j$. Thus, if we have
equality in \eqref{triangle inequality} for such a minimal
representation $Q_{\bar{b}}=\sum_{j\in J} Q_{b_j}$, then
$v_{i,j}\in Q_{b_j}$ and $v_{i,j}\red z_i$ for all occurring
$v_{i,j}$, and we are done.

(It should be noted that we have required $b_j\in G\setminus\{0\}$
for all appearing $b_j$, that is in particular, $b_j\neq 0$. Those
$b_j$ will be sufficient to generate all $\red$-minimal elements
in the extended fiber $Q_{\bar{b}}$. We get the remaining elements in
$Q_{\bar{b}}$ by adding elements from $Q_{0}$.)

Therefore, let us assume that
\begin{equation}\label{inequality}
\sum_{i=1}^k\sum_{j\in J}\|v_{i,j}\|_1 > \sum_{i=1}^k \|z_i\|_1.
\end{equation}
In the following we construct a new representation
$Q_{\bar{b}}=\sum_{j'\in J'} Q_{b_j'}$ and elements $v'_{i,j}$ whose corresponding
sum \eqref{sum} is smaller than the minimally chosen sum. This
contradiction proves that we have indeed equality in
\eqref{triangle inequality} and our claim is proved.

From \eqref{inequality} we conclude that there are indices
$i_0,j_1,j_2$ and a component $m\in\{1,\ldots,n\}$ such that
$v_{i_0,j_1}^{(m)}\cdot v_{i_0,j_2}^{(m)}<0$. As $b_{j_1}, b_{j_2} \in
G$, the sum $b_{j_1} + b_{j_2}$ has been built and the extended fiber
$Q_{b_{j_1}+b_{j_2}}$ has either been reduced to $Q_{0}$ by sets
$Q_{b_{j''}}$, $j''\in J''$, during the Algorithm \ref{Normal form
  algorithm} or $b_{j_1}+b_{j_2}$ has been added to $G$. In the latter
case we set $J'':=\{j''\}$ with $b_{j''} := b_{j_1}+b_{j_2}$. This
gives representations   
\[
v_{i,j_1}+v_{i,j_2}=\sum_{j''\in J''} w_{i,j''} \;\text{ with } \; w_{i,j''}\in
Q_{b_{j''}} \; \text{ and } \; w_{i,j''}\red v_{i,j_1}+v_{i,j_2}
\]
for $i=1,\ldots,k$. As all $w_{i,j''}$ lie in the same orthant of
$\R^n$ as $v_{i,j_1}+v_{i,j_2}$, we get
\[
\biggl \lVert \sum_{j''\in J''} w_{i,j''}\biggr \rVert_1 = \|v_{i,j_1}+v_{i,j_2}\|_1 \leq
\|v_{i,j_1}\|_1+\|v_{i,j_2}\|_1,
\]
with strict inequality for $i=i_0$.

Thus, replacing in $Q_{\bar{b}}=\sum_{j\in J} Q_{b_j}$ the
term $Q_{b_{j_1}}+Q_{b_{j_2}}$ by $\sum_{j''\in J''}
Q_{b_{j''}}$, we arrive at a new representation
$Q_{\bar{b}}=\sum_{j'\in J'} Q_{b_{j'}}$ whose
corresponding sum \eqref{sum} is at most
\[
\sum_{i=1}^k\sum_{j'\in J'}\|v_{i,j'}\|_1 < \sum_{i=1}^k\sum_{j\in
J}\|v_{i,j}\|_1,
\]
contradicting the minimality of the representation
$Q_{\bar{b}}=\sum_{j\in J} Q_{b_j}$. This concludes the
proof. 
\end{proof}

\begin{Remark}
\label{Extended fibers on a lattice}
One may of course use Algorithm \ref{Algorithm to compute atomic
  fibers} also for the problem of finding indecomposable extended
fibers among the elements in the family of 
extended fibers $Q_{b}=\{z:Az=b,z\in\Z^n\}$ where $A$ is kept
fixed and where $b$ is allowed to vary on the lattice which is spanned
by the columns of matrix $A$. The input set then becomes $F = \{\pm
A_1, \ldots, \pm A_n\}$.
\end{Remark}

Having an algorithm available that computes all extended atomic
fibers w.r.t.~a given lattice $\Lambda$, we can use it to
compute partially extended atomic fibers w.r.t.~$\oplus$ and
$\Lambda$: If  
$Q^{(k)}_{b}$ is atomic then so is $Q_{b}$, as any decomposition
of $Q_{b}$, restricted to $\Z^k_+ \times \Z^{n-k}$, would give
a decomposition of $Q^{(k)}_{b}$. This way of computing partially
extended atomic fibers of
a given matrix $A \in \Z^{m \times n}$ is illustrated in Figure
\ref{completion-procedure and sorting out} and formalized in
Algorithm \ref{Dropping reducible}.

\begin{algorithm}[ht]
\caption{Computing partially extended atomic fibers}
\label{Dropping reducible}
\begin{algorithmic}[1]
\REQUIRE{$F:=\{\pm b_1, \ldots, \pm b_s\}$ with $\langle b_1, \ldots,
  b_s \rangle = \Lambda \cap A\Z^n$, $k \in \Z_+$}
\ENSURE{A set $G^*$ such that $\{Q_b^{(k)} \colon b \in G^* \}$
contains all partially extended fibers of order $k$ which are atomic
w.r.t.~$\oplus$ and $\Lambda$}
\STATE Apply Algorithm \ref{Algorithm to compute atomic fibers} to
the set $F$. Let $G$ denote the output.
\STATE $G^*:= \emptyset$.
\FOR{$b \in G$ with $Q_b^{(k)} \neq \emptyset$}
\IF{$Q_b^{(k)} \neq Q_{g}^{(k)} \oplus Q_{b-g}^{(k)}$ for all $g \neq
  b \in G$}
\STATE{$G^* := G^*  \cup  \{b\}$}
\ENDIF
\ENDFOR
\STATE \textbf{return } $G^*$
\end{algorithmic}
\end{algorithm}

\begin{figure}[ht]
\centering
\ifpdf
    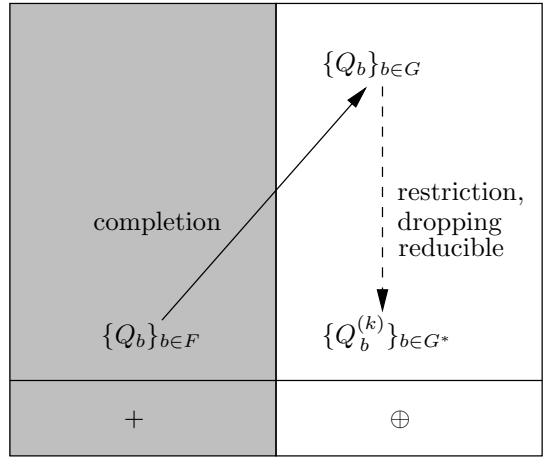
    \else
    \input{normaler-algorithmus.pstex_t}
    \fi
\caption{Computing (partially extended) atomic fibers via extended atomic fibers}
\label{completion-procedure and sorting out}
\end{figure}

The solid arrow from the bottom up in Figure \ref{completion-procedure
  and sorting out} stands for the 
completion procedure which is given by Algorithm \ref{Algorithm to
  compute atomic fibers}. The dashed arrow from the top to 
the bottom illustrates the procedure of intersecting the extended
atomic fibers with $\Z^k_+ \times \Z^{n-k}$ and dropping the reducible (or
empty) fibers afterwards.

Being given the atomic fibers $P_b$ of a matrix $A$ it is easy to
compute the atomic fibers $\tilde{P}_b$ which have been defined in
\cite{Adams+Hosten+Loustaunau+Miller:99}. Recall that
$\tilde{P}_b := \conv \{z \,:\, Az=b \; z \in \Z^n_+\}$ and that
$\tilde{P}_b$ is said to be atomic if each decomposition $\tilde{P}=
\tilde{P}_{b_1}+\tilde{P}_{b_2}$ implies either $b= b_1$ or $b=b_2$.

\begin{Lemma}
\label{convex hulls and atomic fibers}
If $\tilde{P}_b$ is an atomic fiber of the matrix $A$ then $P_b$ is
atomic, too.
\end{Lemma}
\begin{proof}
Suppose $P_b = P_{b_1} + P_{b_2}$ (and $b_1, b_2 \neq 0$). Then we
have: $\tilde{P}_b = \conv (P_b) = \conv (P_{b_1}+P_{b_2}) =
\conv(P_{b_1}) + \conv(P_{b_2}) = \tilde{P}_{b_1} + \tilde{P}_{b_2}$
which is a contradiction.
\end{proof}

Lemma \ref{convex hulls and atomic fibers} enables us to compute the
atomic fibers $\tilde{P}_b$ via Algorithm \ref{compute convex atomic
  fibers}. 

\begin{algorithm}[ht]
\caption{Computing the atomic fibers $\tilde{P}_b$}
\label{compute convex atomic fibers}
\begin{algorithmic}[1]
\REQUIRE{$F:=\{b_1, \ldots, b_s\}$ with $P_{b_i}$ is an
  atomic fiber}
\ENSURE{A set $G = \{\bar{b}_1, \ldots, \bar{b}_t\}$ such that
  $\tilde{P}_{\bar{b}_i}$ is an atomic fiber.}
\STATE Set $G:= \emptyset$.
\FORALL {$b \in F$}
\IF{$\tilde{P}_b \neq \tilde{P}_g + \tilde{P}_{b -g}$ for all $b \neq
  g \in F$}
\STATE $G:= G \cup \{b\}$
\ENDIF
\ENDFOR
\STATE \textbf{return } $G$
\end{algorithmic}
\end{algorithm}

\section{Preliminaries of the project-and-lift algorithm}
\label{Section: Preliminaries of the project-and-lift algorithm}

The way of computing atomic
fibers presented in section \ref{Section: Computation of Atomic
  Fibers}, however, is pretty slow, since there are far more extended
atomic fibers than atomic fibers. A similar behavior can be observed
when one extracts the Hilbert basis of the cone
$\{x:Ax=0,x\in\R^n_+\}$ from the Graver basis of $A$, as the
Graver basis is usually much bigger than the Hilbert basis one is
interested in.
\cite{Hemmecke:2003} showed that one can reduce this
difference in sizes by a project-and-lift algorithm.
With this algorithm,
bigger Hilbert bases, even with more than $500{,}000$ elements, can
be computed nowadays.

In this section and in the following one, we will present a similar
algorithm to compute the atomic fibers of a given matrix $A \in \Z^{d
  \times n}$ which is
significantly faster than Algorithm \ref{Dropping reducible}. This
algorithm puts us in the position
to compute not only the atomic fibers of a matrix but the atomic
fibers w.r.t.~an arbitrary (finitely generated monoid) $M$, i.e., 
the right-hand side $b$ is only allowed to vary in this monoid.
During the algorithm we consider partially extended fibers
$Q_{b}^{(k-1)} = \{z \in \Z^{k-1}_+ \times \Z^{n-k+1} \colon Az = b\}$ with
varying $b \in M$ w.r.t.~$k$-restricted Minkowski-sums.

Let $M = \langle m_1, \ldots, m_t \rangle \subseteq \Z^d$ be a
finitely generated 
monoid and let $A \in \Z^{d \times n}$ be a matrix. We want to
compute the atomic fibers of matrix $A$ w.r.t.~the monoid $M$. The
algorithm proceeds in $n$ individual steps. The $k$-th step is
illustrated in Figure \ref{k-th step of project-and-lift}.

\begin{figure}[ht]
\centering
\ifpdf
    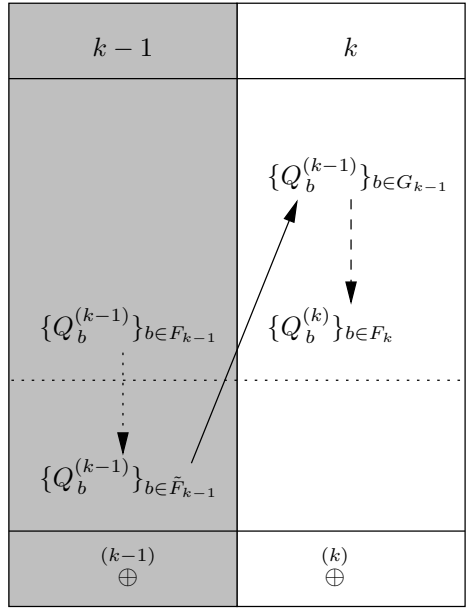
    \else
    \input{k-th_step.pstex_t}
    \fi
\caption{The $k$-th step of the project-and-lift algorithm}
\label{k-th step of project-and-lift}
\end{figure}

The $k$-th lifting step follows the arrows in the figure. 
It starts by performing a ``preprocessing step'' in which the
input set is prepared for the main part of this lifting step. This
process is illustrated by the dotted arrow and will be explained in
more detail in section \ref{preprocessing the input set}.

The $k$-th lifting step continues as follows: it performs a
completion step similar to the one we presented in Algorithm
\ref{Algorithm to compute atomic fibers}, which is illustrated by the
solid arrow going from the bottom up. This
step will be explained in more detail in section \ref{the completion
  procedure}. 

The dashed arrow, finally, stands for a step where we drop all
elements of the fibers having a negative $k$-th component. 
It might
happen that an atomic partially extended fiber becomes empty or
reducible when processing this last step. Therefore we have to
perform another reducibility test. The details of this subroutine
will be given in section 
\ref{the sorting and reducing step}. 

Having performed the $k$-th lifting
step we continue performing the $(k+1)$-st lifting step. The whole
project-and-lift algorithm is illustrated in Figure 
\ref{project-and-lift-scheme}. After having performed $n$ of these
lifting steps we arrive at the finitely many fibers of the matrix $A$
which are atomic w.r.t.~$M$. 

\begin{figure}[h]
\centering
\ifpdf
    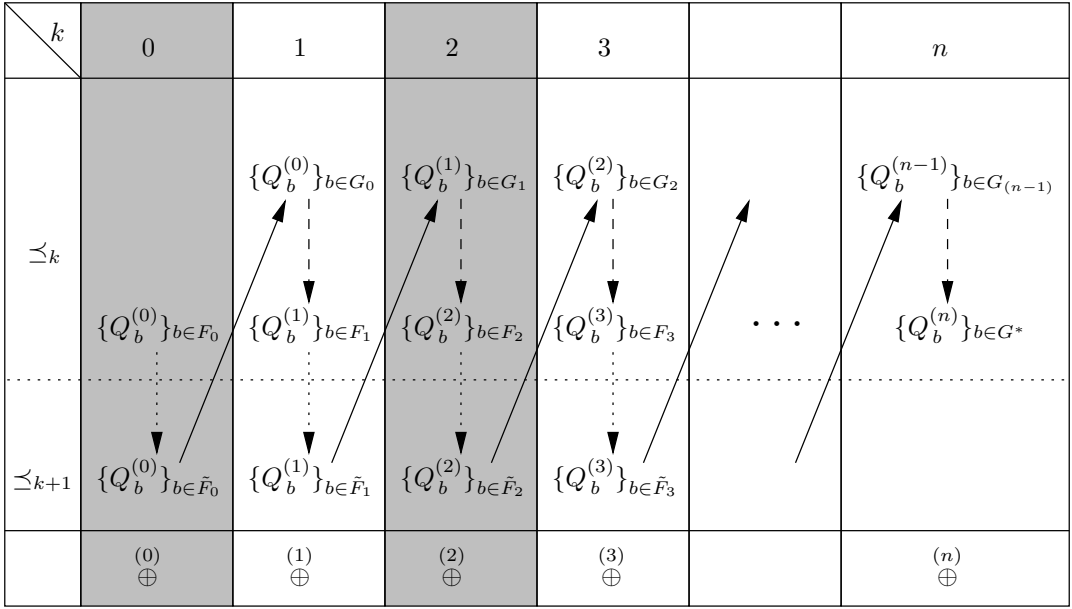
    \else
    \input{project-and-lift-scheme.pstex_t}
    \fi
\caption{The scheme of the project-and-lift algorithm}
\label{project-and-lift-scheme}
\end{figure}    

\subsection*{Dealing with infinitely many atomic fibers}
Let $A \in \Z^{d \times n}$ be a matrix and $M \subseteq \Z^d$ a
monoid which is finitely generated.
The project-and-lift algorithm will deal with partially
extended fibers w.r.t.~$\mathop{\oplus}^{(l)}$ and the monoid $M$
where $l \leq n$. Recall from Definition \ref{defining atomic fibers}
that $Q_b^{(k)}$ is atomic w.r.t.~$\mathop{\oplus}^{(l)}$ and $M$ if
there is no decomposition 
\begin{displaymath}
Q_{b}^{(k)} = Q_{b_1}^{(k)} \koplus{(l)} Q_{b_2}^{(k)}
\end{displaymath} 
with $b_1, b_2 \in M$ and $\pi_l(Q_{b_1}^{(k)}),\pi_l(Q_{b_2}^{(k)})
\neq \pi_l(Q_0^{(k)})$. 
Note that for $l < n$ there are usually some
$\bar{b} \in M$ with $\pi_l(Q_{\bar{b}}^{(k)}) =
\pi_l(Q_0^{(k)})$. Therefore, if $Q_b^{(k)}$ is atomic
w.r.t.~$\mathop{\oplus}^{(l)}$ and $M$ then so is
$Q_{b+\bar{b}}^{(k)}$, $Q_{b+2\bar{b}}^{(k)}$, \dots. This means that
for $l < n$ we usually have infinitely many partially extended
fibers which are atomic w.r.t.~$\mathop{\oplus}^{(l)}$. It is clear that no terminating algorithm may compute the
whole set of atomic partially extended fibers
w.r.t.~$\mathop{\oplus}^{(l)}$ and $M$. Therefore we introduce a
preorder $\preceq_l$ (i.e., a reflexive and transitive binary relation) on the
set of right-hand side vectors $b 
\in M$ with non-empty partially extended fiber $Q_b^{(k)}$ and perform
the $l$-th step of the project-and-lift algorithm 
w.r.t.~the preorder $\preceq_l$.

\begin{Definition}
\label{partial ordering}
Let $M^{(k)} \subseteq M$ be the submonoid of $M$ with $Q_b^{(k)}
\neq \emptyset$ for $b \in M^{(k)}$.
Let $A \in \Z^{d \times n}$, $0 \leq k \leq
l \leq n$ and let $b, \bar{b} \in M^{(k)}$. We say that $b \preceq_l  
\bar{b}$ if $\bar{b} -b \in \bar{S}^{(l)}$, where
$\bar{S}^{(l)} = \{\lambda_{l+1}A_{l+1} + \ldots + \lambda_nA_n
\colon \lambda_i \in \Z\} \cap M$. 

$b \in M^{(k)}$ is called \emph{minimal
w.r.t.~$\preceq_l$} if there is no 
$b \neq \bar{b} \in M^{(k)}$ with
$\bar{b} \preceq_l b$. 

Note that $\bar{b} \preceq_l b$ implies
$\pi_l(Q_{\bar{b}}^{(k)}) = \pi_l(Q_b^{(k)})$ and $Q_b^{(k)} =
Q_{\bar{b}}^{(k)} \mathop{\oplus}^{(l)} Q_{b -\bar{b}}^{(k)}$.
\end{Definition}

The relation $\preceq_l$ defines a preorder on the set of
right-hand sides $b \in M$ with non-empty partially extended fibers
of order $k$. Additionally we have the following relation between the
sets $\bar{S}^{(l)}$:
\begin{equation}
\label{relation between S}
\{0\} = \bar{S}^{(n)} \subseteq \bar{S}^{(n-1)} \subseteq \ldots
\subseteq \bar{S}^{0} = M^{(0)}.
\end{equation}

\begin{Lemma}
\label{finitely many atomic fibers}
Let $M = \langle m_1, \ldots, m_t \rangle$ be a monoid which is
finitely generated.
\begin{enumerate}
\item[(i)] \addvspace{-8pt}
Let $0 \leq k \leq l\leq n$ and 
let $M \supseteq F = \{b_1, b_2, \ldots \}$ be a set of vectors with $b_i
\npreceq_l b_j$ for all $i < j$.  Then $F$ is finite. 
\item[(ii)]
Let $0 \leq k \leq l \leq n$ and let $F = \{b_1, b_2, \ldots \}$ be a
set of right hand sides satisfying $Q_{b_i}^{(k)}$ is atomic
w.r.t.~$\mathop{\oplus}^{(l)}$ and $M$ and $b_i \npreceq_l b_j$ for
all $b_i \neq b_j$. Then $F$ is finite.
\end{enumerate}
\end{Lemma}

\begin{proof}
(i): 
Let $b_i, b_j \in F$ with $i < j$ and let $\alpha^i, \alpha^j \in
\Z_+^t$ with $b_i = \sum\nolimits_{k=1}^{t}{\alpha^i_km_k}$ and $b_j
= \sum\nolimits_{k=1}^{t}{\alpha^j_km_k}$. Then
$(\alpha^i,Q_{b_i}^{(k)}) \ntrianglelefteq
(\alpha^j,Q_{b_j}^{(k)})$. Suppose not. Then we have $\alpha^i \red
\alpha^j$ and $Q_{b_j}^{(k)} = Q_{b_i}^{(k)} \mathop{\oplus}
Q_{b_j-b_i}^{(k)}$ which implies that $Q_{b_j}^{(k)} = Q_{b_i}^{(k)}
\mathop{\oplus}^{(l)} Q_{b_j-b_i}^{(k)}$. But this last relation
contradicts the fact that $b_i \npreceq_l b_j$. Therefore $(\alpha^i,
Q_{b_i}^{(k)}) \ntrianglelefteq (\alpha^j, Q_{b_j}^{(k)})$ for all
$b_i,b_j \in F$ for $i < j$. Finiteness of $F$ follows with Lemma
\ref{finitely many irreducible pairs} (i). \newline
(ii): This is a direct consequence of (i). \end{proof}

Our algorithm will work with sets of vectors $F$ which
have the property claimed in Lemma \ref{finitely many atomic
  fibers}. Additionally they will admit the following property: if $b
\in M$ is the right-hand side of a partially extended fiber
$Q_b^{(k)}$ which is atomic w.r.t.~$\mathop{\oplus}^{(l)}$ and $M$
then there is $\bar{b} \in F$ with $\bar{b} \preceq_l b$. This means
in particular: If $\bar{b} \in M$ is minimal w.r.t.~$\preceq_l$ and
$Q_{\bar{b}}^{(k)}$ is atomic w.r.t.~$\mathop{\oplus}^{(l)}$ and $M$
then $\bar{b} \in F$.  Note, however, that the converse is not true in
general:  It is not guaranteed that for every $b\in M$ there is a $\bar{b}
\preceq_l b$ that is minimal w.r.t.~$\preceq_l$.

\section{The k-th step of the project-and-lift algorithm}
\label{Section: The k-th step of the project-and-lift algorithm}
In the following subsections we will explain the individual steps the
project-and-lift algorithm performs during one lifting step.

\subsection{The completion procedure}
\label{the completion procedure}

In this
subsection we will explain the so-called ``completion procedure'' in the
$k$-th step of the algorithm. This part is illustrated in Figure
\ref{completion}.

\begin{figure}[ht]
\centering
\ifpdf
    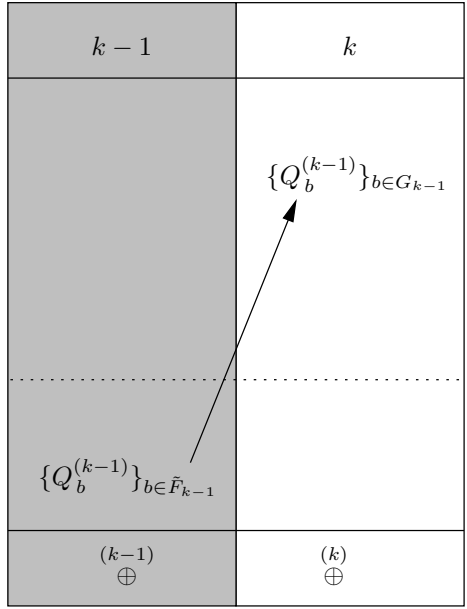
    \else
    \input{completion.pstex_t}
    \fi
\caption{The completion procedure of the $k$-th lifting step}
\label{completion}
\end{figure}

Let $M$ be a monoid which is finitely generated. We
denote by $M^{(k)} = \{b \in M \colon Q_b^{(k)} \neq \emptyset\}$ the
submonoid of all right-hand sides $b \in M$ having non-empty
partially extended fibers of order $k$. We have:
\begin{displaymath}
M \supseteq M^{(0)} \supseteq M^{(1)} \supseteq \ldots \supseteq M^{(n)}.
\end{displaymath}
As in the previous section, $\bar{S}^{(l)}$, $0 \leq l \leq n$, will
denote the set 
$\{\lambda_{l+1}A_{l+1} + \ldots + \lambda_nA_n \colon \lambda_i \in
\Z\} \cap M$.

\begin{Definition}
\label{weight function}
We introduce a weight function $\omega_{m}$ for partially extended
fibers $Q_b^{(k)}$ ($m \leq k \leq n$) by
\begin{displaymath}
\omega_{m}(Q_b^{(k)}) = \mathrm{min }\{||\pi_m(v)||_1 \colon v \in
Q_b^{(k)}\}.
\end{displaymath}
\end{Definition}

\begin{Remark}
Actually, it suffices to determine $||\pi_m(v)||_1$ for
$\red_m$-minimal elements $v$ in $Q_b^{(k)}$ to determine the value of
$\omega_m(Q_b^{(k)})$. To see this, suppose there is $w \in
Q_b^{(k)}$ non-minimal w.r.t.~$\red_m$. Then there is $v \in
Q_b^{(k)}$ with $v \red_m w$ and thus $0 \leq v^{j} \leq w^{j}$ for $j =1,
\ldots, k$. Therefore $||\pi_m(v)||_1 \leq ||\pi_m(w)||_1$.
\end{Remark}

\begin{algorithm}[ht]
\caption{The completion procedure to compute atomic partially
  extended fibers w.r.t.~a monoid}
\begin{algorithmic}[1]
\REQUIRE{A set $\tilde{F}_{k-1} \subseteq M^{(k-1)}$ with the
following properties:
\begin{enumerate}
\item[(i)]
For every right-hand side $b \in M^{(k-1)}$ of a partially
extended fiber $Q_b^{(k-1)}$ which is atomic
w.r.t.~$\mathop{\oplus}^{(k-1)}$ and $M$ there exists $\tilde{b} \in
\tilde{F}_{k-1}$ with $\tilde{b} \preceq_{k} b$.
\item[(ii)]
$b_i \npreceq_k b_j$ for $b_i,b_j \in \tilde{F}_{k-1}$ with $b_i \neq
b_j$.
\end{enumerate}}
\ENSURE{A set $G_{k-1} \subseteq M^{(k-1)}$ with the
properties: 
\begin{enumerate}
\item[(i)]
For every right-hand side $b \in M^{(k-1)}$ of a partially
extended fiber $Q_b^{(k-1)}$ which is atomic
w.r.t.~$\mathop{\oplus}^{(k)}$ and $M$ there exists $\tilde{b} \in
G_{k-1}$ with $\tilde{b} \preceq_{k} b$. 
\item[(ii)]
$b_i \npreceq_{k} b_j$ for $b_i \neq b_j \in
G_{k-1}$.
\end{enumerate}}
\STATE $\bar{G}^{\omega =0}:= \{f \in \tilde{F}_{k-1} \colon \omega_{k-1}(Q_f^{(k-1)})
= 0\}$
\STATE $C^{\omega=0}:= \bigcup\nolimits_{f,g \in \bar{G}^{\omega = 0}}{\{f+g\}}$
\WHILE{$C^{\omega=0} \neq \emptyset$}
\STATE $s:=$ an element in $C^{\omega =0}$ 
\STATE $C^{\omega=0}:=C^{\omega=0} \setminus \{s\}$
\STATE $f:=$ monoid-normal-form $(s,\bar{G}^{\omega =0},\emptyset,M^{(k-1)})$
\IF{$f \notin \bar{S}^{(k)}$}
\STATE $\bar{G}^{\omega =0}:= \bar{G}^{\omega =0} \cup \{f\}$
\STATE $C^{\omega=0}:= C^{\omega =0} \cup \bigcup\nolimits_{g \in \bar{G}^{\omega =
    0}}{\{f+g\}}$
\ENDIF
\ENDWHILE
\STATE $G^{\omega=0} := \emptyset$
\FORALL{$b \in \bar{G}^{\omega=0}$}
\IF {$Q_b^{(k-1)} \neq Q_g^{(k-1)} \mathop{\oplus}^{(k)} Q_{b-g}^{(k-1)}$
for all $b \neq g \in \bar{G}^{\omega=0}$ with $b-g \in M$}
\STATE $G^{\omega=0} := G^{\omega=0} \cup \{b\}$
\ENDIF
\ENDFOR
\STATE  $\bar{G}^{\omega \geq 1} := \{f \in \tilde{F}_{k-1} \colon \omega_{k-1}(Q_f^{(k-1)})
> 0\}$, $G^{\omega \geq 1} = \emptyset$
\FORALL{ $g \in \bar{G}^{\omega \geq 1}$}
\STATE $G^{\omega \geq 1} = G^{\omega \geq 1} \cup$
monoid-normal-form $(g,G^{\omega=0}, \emptyset, M^{(k-1)})$
\ENDFOR
\STATE $C^{\omega \geq 1}:=\bigcup\nolimits_{f,g \in G^{\omega
    \geq 1}}{\{f+g\}}$
\WHILE {$C^{\omega \geq 1} \neq \emptyset$}
\STATE $s :=$ an element in $C^{\omega \geq 1}$ with smallest weight
$\omega_{k-1}(Q_s^{(k-1)})$
\STATE $C^{\omega \geq 1} := C^{\omega \geq 1} \setminus \{s\}$
\STATE $f:=$ monoid-normal-form $(s,G^{\omega=0},G^{\omega \geq
  1},M^{(k-1)})$
\IF {$f \neq 0$}
\STATE $G^{\omega \geq 1} := G^{\omega \geq 1} \cup \{f\}$
\STATE $C^{\omega \geq 1} := C^{\omega \geq 1} \cup
\bigcup\nolimits_{g \in G^{\omega = 0} \cup G^{\omega \geq 1}}{\{f+g\}}$
\ENDIF
\ENDWHILE
\STATE $G_{k-1}:=G^{\omega =0} \cup G^{\omega \geq 1} \cup \{0\}$
\STATE \textbf{return } $G_{k-1}$
\end{algorithmic}
\label{m-pl}
\end{algorithm}

\begin{algorithm}
\caption{The monoid-normal-form algorithm}
\begin{algorithmic}[1]
\REQUIRE{$s, G^{\omega =0}, G^{\omega \geq 1}$, membership oracle for $M^{(k-1)}$}
\ENSURE{a normal form of $s$ w.r.t.~$G^{\omega =0} \cup G^{\omega
    \geq 1}$ and $M^{(k-1)}$}
\IF{$\exists g \in G^{\omega \geq 1}$ with $Q_{s}^{(k-1)}=
  Q_{g}^{(k-1)} \mathop{\oplus}^{(k)} Q_{s-g}^{(k-1)}$ and
  $s-g \in M^{(k-1)}$}
\STATE \textbf{return } $0$
\ELSE
\WHILE {$\exists g \in G^{\omega =0}$ with $Q_{s}^{(k-1)}=
  Q_{g}^{(k-1)} \mathop{\oplus}^{(k)} Q_{s-g}^{(k-1)}$ and
  $s -g \in M^{(k-1)}$}
\STATE $s:=s-g$
\ENDWHILE
\STATE \textbf{return } $s$
\ENDIF
\end{algorithmic}
\label{m-nF}
\end{algorithm}

\begin{Lemma}
\label{correctness of m-pl}
Algorithm \ref{m-pl} with input set $\tilde{F}_{k-1}:= \{b_1, \ldots,
b_s\}$ and monoid $M^{(k-1)}$ terminates and computes a set
$G_{k-1} = G^{\omega=0} \cup G^{\omega \geq 1} \cup \{0\}
\subseteq M^{(k-1)}$ with properties (i) and (ii). 
\end{Lemma}

For the proof of Lemma~\ref{correctness of m-pl}, we have to introduce some
more notation. 
\begin{Notation}
\label{elements partially ext fibers}
During the proof of Algorithm \ref{m-pl} we examine elements of
partially extended fibers. These elements will be denoted as follows: 
$$Q_{b}^{(k-1)} \ni z= (z_1,z_2,z_3) \in \Z^{k-1}_+ \times \Z \times
\Z^{n-k},$$
i.e., $z_1 \in \Z^{(k-1)}_+$ denotes the first $k-1$ components, $z_2 \in \Z$ the
$k$-th component and $z_3 \in \Z^{(n-k)}$ denotes the last
$n-k$ components.
\end{Notation}

We will use the following lemma in the proof of Lemma~\ref{correctness of m-pl}.
\begin{Lemma}
\label{Representation}
Let $\tilde{F}_{k-1} \subseteq M^{(k-1)}$ be a set admitting the
following property: for every
right-hand side $b \in M^{(k-1)}$ of a partially 
extended fiber $Q_b^{(k-1)}$ which is atomic
w.r.t.~$\mathop{\oplus}^{(k-1)}$ and $M$ there exists $\tilde{b} \in
\tilde{F}_{k-1}$ with $\tilde{b} \preceq_{k} b$. 
Let $\beta \in
M^{(k-1)}$ be the right-hand side of an arbitrary partially extended
fiber of order $(k-1)$. Then we find $\tilde{b}_i \in \tilde{F}_{k-1}$ such that for
$M^{(k-1)} \owns \tilde{\beta} := \sum\nolimits{\tilde{b}_i}$ we have $\tilde{\beta}
\preceq_k \beta$ and 
\begin{equation}
\label{representation of Q}
Q_{\tilde{\beta}}^{(k-1)} = \bigoplus\nolimits_{i \in
  I}^{(k-1)}{Q_{\tilde{b_i}}^{(k-1)}}.
\end{equation}
\end{Lemma}

\begin{proof}
Let $\beta \in M^{(k-1)}\setminus \{0\}$ be the right-hand side of a partially
extended fiber of order $(k-1)$. Consider a decomposition of
$Q_{\beta}^{(k-1)}$ into a sum of fibers 
of order $(k-1)$ which are atomic w.r.t.~$\mathop{\oplus}^{(k-1)}$
and $M$:

\begin{displaymath}
Q_{\beta}^{(k-1)} = \bigoplus\limits_{i \in I}^{(k-1)}{Q_{b_i}^{(k-1)}}
\end{displaymath}

As the partially extended fibers $Q_{b_i}^{(k-1)}$ are atomic
w.r.t.~$\mathop{\oplus}^{(k-1)}$ and $M$  there are
$\tilde{b}_i \in \tilde{F}_{k-1}$  with $\tilde{b}_i \preceq_k b_i$ for all $i \in
I$. Consider $M^{(k-1)} \owns \tilde{\beta} := \sum\nolimits_{i \in I}{\tilde{b}_i}$. We have
$\tilde{\beta} \preceq_k \beta$ because $b_i - \tilde{b_i} \in \bar{S}^{(k)}$
implies $\sum\nolimits_{i
  \in I}{b_i - \tilde{b}_i} = \beta -\tilde{\beta} \in \bar{S}^{(k)}$. 
Additionally 
\begin{equation}
\pi_{k-1}(Q_{\tilde{\beta}}^{(k-1)}) = \pi_{k-1}(Q_{\beta}^{(k-1)}) =
\bigoplus\nolimits_{i \in I}^{(k-1)}{\pi_{k-1}(Q_{b_i}^{(k-1)})}
 = \bigoplus\nolimits_{i \in I}^{(k-1)}{\pi_{k-1}(Q_{\tilde{b}_i}^{(k-1)})},
\end{equation}
which together with $\tilde{\beta} =\sum\nolimits_{i \in I}{\tilde{b}_i}$ implies
that $Q_{\tilde{\beta}}^{(k-1)} = \bigoplus\limits_{i \in
  I}^{(k-1)}{Q_{\tilde{b}_i}^{(k-1)}}$ and our claim is proved. 
\end{proof}  

We are now in the position to proof Lemma~\ref{correctness of m-pl}.

\begin{proof}[Proof of Lemma~\ref{correctness of m-pl}]
As the following proof will be slightly complex, consider the
following outline of the proof first. 
\begin{enumerate}
\item We show that $G_{k-1} \subseteq M^{(k-1)}$.
\item We show that the set $G^{\omega = 0}$ is finite and that for
  all $b_i, b_j \in G^{\omega = 0}$ with $b_i \neq b_j$ we have $b_i
  \npreceq_k b_j$.
\item We show that $G_{k-1}$ is finite. This implies that Algorithm
  \ref{m-pl} terminates. At the same time 
  we show that the output set admits property (ii), i.e.,  $b_i
\npreceq_k b_j$ for $b_i \neq  b_j \in G_{k-1}$.
\item 
 We show that if $Q_{b}^{(k-1)}$ is an atomic partially extended
  fiber w.r.t.~$\mathop{\oplus}^{(k)}$ and $M$  then there is $\tilde{b}
  \preceq_{k} b$ with $\tilde{b} \in G_{k-1}$. This is property (i) of
  the output set.  
\end{enumerate}

\textit{Step 1.} \newline
It is clear that Algorithm \ref{m-pl} returns a set $G_{k-1}
\subseteq M^{k-1}$. This is guaranteed by the
monoid-normal-form algorithm, where we ensure that the elements added
lie in $M^{(k-1)}$.

\textit{Step 2.} \newline
We will now prove that the set $G^{\omega =0} $ is finite. 
To this aim we show finiteness of $\bar{G}^{\omega=0}$ first. Consider
the sequence $\bar{G}^{\omega =0} \setminus \{f \in \tilde{F}_{k-1} \, : \,
\omega_{k-1}(Q_f^{(k-1)}) = 0\} = \{f_1,f_2,\ldots \}$  produced in
lines 1--11 of the algorithm. Clearly $f_i \in 
M^{(k-1)}$ for all $i$. Additionally $f_i \npreceq_k f_j$ for all
$i < j$. Suppose not and let $f_i \preceq_k f_j$. Then
$Q_{f_j}^{(k-1)}=Q_{f_i}^{(k-1)} \mathop{\oplus}^{(k)}
Q_{f_j-f_i}^{(k-1)}$. As $f_j$ has been added to $\bar{G}^{\omega
  =0}$, the second criterion of 
the monoid-normal-form algorithm is not satisfied, i.e., $f_j-f_i
\notin M^{(k-1)}$. But $f_j-f_i \in
\bar{S}^{(k)}$ implies in particular that $f_j-f_i \in
M$ and as $Q_{f_j-f_i}^{(k-1)} \neq \emptyset$ we have $f_j - f_i \in
M^{(k-1)}$ which is a contradiction. 
Therefore $f_i \npreceq_k f_j$ for all
 $f_i,f_j \in \{f_1,f_2, \ldots \}$ with $i < j$. Finiteness of
 $\bar{G}^{\omega =0}$ follows with Lemma \ref{finitely many atomic
   fibers}. 

As $G^{\omega=0} \subseteq \bar{G}^{\omega =0}$ it is clear now that
$G^{\omega=0}$ is finite. Additionally lines 12--17 of Algorithm
\ref{m-pl} guarantee that $b_i \npreceq_k
b_j$ for all $b_i,b_j \in G^{\omega = 0}$ with $b_i \neq b_j$.

\textit{Step 3.} \newline
 Let $G^{\omega = \alpha} := \{b \in G_{k-1} \, : \,
\omega_{k-1}(Q_b^{(k-1)}) = \alpha\}$ for $\alpha \in
\Z_+$. Furthermore let $G^{\omega \leq \alpha} := \{b \in G_{k-1} \, : \,
\omega_{k-1}(Q_b^{(k-1)}) \leq \alpha\}$ for $\alpha \in \Z_+$. We
will show via induction that 
that $b_i \npreceq_k b_j$ for $b_i,b_j \in G^{\omega \leq \alpha}$
with $b_i \neq b_j$. Lemma \ref{finitely many atomic fibers} then
yields that $G^{\omega \leq \alpha}$ is finite.
Clearly we have $G_{k-1} =
\bigcup\nolimits_{\alpha\in \Z_+}{G^{\omega \leq \alpha}}$. Let $b_i,
b_j \in G_{k-1}$ with $b_i \neq b_j$. Then there is $\alpha \in \Z_+$
with $b_i,b_j \in G^{\omega \leq \alpha}$ yielding $b_i \npreceq_k
b_j$. The set $G_{k-1}$ admits property $(ii)$ of the output set thus
which together with  Lemma \ref{finitely many atomic fibers} yields
that $G_{k-1}$ is finite.

% We will then show that every element $b \in G^{\omega \leq \alpha}$ is
% the right-hand side of an atomic partially extended fiber
% $Q_b^{(k-1)}$ w.r.t.~$\mathop{\oplus}^{(k)}$ and $M$. This
% intermediate result will help us to show that $G_{k-1}$ is finite.

We will show via induction that $G^{\omega \leq \alpha}$ is finite.
With step 2 of the proof we know that
our claim is proved for $\alpha = 0$. Suppose that our assertions are
true for all integers lower or equal than $\alpha$. We will prove our
claim for $\alpha +1$. Let $b_i, b_j \in G^{\omega \leq \alpha
  +1}$ and suppose  $b_i \preceq_k b_j$. There are several
cases possible:
\begin{enumerate}
\item[(i)]
$b_i, b_j \in \tilde{F}_{k-1}$ \newline
This contradicts input property $(ii)$ of the input set
$\tilde{F}_{k-1}$. 
\item[(ii)]
$b_i \in \tilde{F}_{k-1}$, $b_j \notin \tilde{F}_{k-1}$ \newline
This contradicts the if-clause of Algorithm \ref{m-nF} because $b_i$
then is an appropriate reducer of $b_j$. 
\item[(iii)]
$b_i \notin \tilde{F}_{k-1}$, $b_j \in \tilde{F}_{k-1}$ \newline
As $\omega_{k-1}(Q_{b_j -b_i}^{(k-1)}) = 0$ and as $G^{\omega = 0}$
is completed before $G^{\omega \geq 1}$ we know that there is
$\bar{b} \in G^{\omega = 0}$ with $\bar{b} \preceq_{k} b_j-b_i
\preceq_k b_j$. But this is a contradiction to lines 18-21 of
Algorithm \ref{m-pl} because in this case $b_j$ would not have been added to
$G^{\omega \geq 1}$ then.
\item[(iv)] $b_i, b_j \notin \tilde{F}_{k-1}$ \newline
Depending on whether either $b_i$ has been added to $G^{\omega =
  \alpha + 1}$ befre $b_j$ was added or not we either have a
contradiction to the if-clause of Algorithm \ref{m-nF} or to the
else-clause of this algorithm.
\end{enumerate}

We know via induction that $G^{\omega \leq \alpha}$ admits property
$(ii)$ of the output set. We will now show that this is also true for
$G^{\omega \leq \alpha +1}$. Let $b_i, b_j \in G^{\omega \leq \alpha
  +1}$ and suppose $b_i \preceq_k b_j$. By induction, the previous
discussion and monotonicity of the weight-function $\omega_k(\cdot)$:
$b_i \in G^{\omega \leq \alpha}$ and $b_j \in G^{\omega =
  \alpha+1}$. But this contradicts the if-clause of Algorithm
\ref{m-nF}. Therefore $G^{\omega \leq \alpha+1}$ admits property
$(ii)$ of the output set which had to be proved.

\textit{Step 4.} \newline
Let $b \in M^{(k-1)}$ such that $Q_{b}^{(k-1)}$ is atomic
with respect to $\mathop{\oplus}^{(k)}$ and $M$. With
Lemma~\ref{Representation} we 
know that there is $\tilde{b} \in M$, $\tilde{b} \preceq_k b$,
admitting a representation 
\eqref{representation of Q}:
\begin{displaymath}
Q_{\tilde{b}}^{(k-1)} = \bigoplus\nolimits_{i \in
  I}^{(k-1)}{Q_{b_i}^{(k-1)}},
\end{displaymath}
where $b_i \in \tilde{F}_{k-1}$. We will show that $\tilde{b} \in
G_{k-1}$. The above representation implies in
particular that every $z=(z_1,z_2,z_3) \in
 Q_{\tilde{b}}^{(k-1)}$ can be
 written as $(z_1,z_2,z_3) = \sum\nolimits_{i \in I}{(z^i_1,z^i_2,z^i_3)}$ with
 $(z^i_1,z^i_2,z^i_3) \in Q_{b_i}^{(k-1)}$. In particular: 
\begin{equation}
\label{kleiner_gleich}
(z^i_1,z^i_2,z^i_3) \red_{k-1} (z_1,z_2,z_3)
\end{equation}
 for all $i$. If $Q_{b_i}^{(k-1)} \ni (z_1^i,z_2^i,z_3^i)
 \sqsubseteq_{k} (z_1,z_2,z_3)$ was valid this then would imply:
 $Q_{\tilde{b}}^{(k-1)} = 
 \bigoplus\nolimits_{i \in I}^{(k)}{Q_{b_i}^{(k-1)}}$. 

Let $R_{\tilde{b},k}^{(k-1)} =
\{(\bar{z}^1_1,\bar{z}^1_2,\bar{z}^1_3), \ldots,
(\bar{z}_1^t,\bar{z}_2^t,\bar{z}_3^t)\}$ be the set of representatives
of the $\red_k$-minimal elements in $Q_{\tilde{b}}^{(k-1)}$ according
to Definition \ref{minimal elements}. With Lemma \ref{minimal
  elements suffice} we know that it is sufficient to analyze the
$\red_k$-minimal elements in a partially extended fiber to decide
decomposability w.r.t.~$\mathop{\oplus}^{(k)}$.

From 
all representations $Q_{\tilde{b}}^{(k-1)} = 
\bigoplus\nolimits_{j \in J}^{(k-1)}{Q_{b_j}^{(k-1)}}$ with $b_j \in
G_{k-1}$ and $\pi_k(Q_{b_j}^{(k-1)}) \neq \pi_k(Q_{0}^{(k-1)})$
and where the $\red_{k-1}$-minimal elements in 
$R_{\tilde{b},k}^{(k-1)}$ are represented as
$(\bar{z}_1^i,\bar{z}_2^i,\bar{z}_3^i) = \sum\nolimits_{j \in 
  J}{(z_1^{i,j},z_2^{i,j},z_3^{i,j})}$ with $(z^{i,j}_1,z_2^{i,j},z_3^{i,j}) \in
Q_{b_j}^{(k-1)}$ for $i=1, \ldots, t$, choose a representation and
elements $(z^{i,j}_1,z_2^{i,j},z_3^{i,j})$ such that the sum
\begin{equation}
\label{sum_elements}
\sum\limits_{i=1}^{t}{\sum\limits_{j \in J}{||(z_1^{i,j},z_2^{i,j})||_1}}
\end{equation}
is minimal. 

% Note that in every such representation of
% $Q_{\tilde{b}}^{(k-1)}$ the summands admit a weight
% $\omega_{k-1}(Q_{b_j}^{(k-1)}) = 0$. 

By the triangle inequality we have
\begin{equation}
\label{norm_inequality}
\sum\limits_{i=1}^{t}{\sum\limits_{j \in J}{||(z_1^{i,j},z_2^{i,j})||_1}}
\geq \sum\limits_{i=1}^{t}{||(\bar{z}_1^i,\bar{z}_2^i)||_1}
\end{equation}
Herein equality holds if and only if all $(z_1^{i,j},z_2^{i,j})$ have the
same sign pattern as $(\bar{z}_1^i,\bar{z}_2^i)$, $i=1, \ldots, t$,
that is if and only if we
have $(z_1^{i,j},z_2^{i,j},z_3^{i,j}) \sqsubseteq_{k}
(\bar{z}_1^i,\bar{z}_2^i,\bar{z}_3^i)$ for all $j \in J$ and all 
$i=1, \ldots, t$. Thus if we have equality in \eqref{norm_inequality}
for such a minimal representation $Q_{\tilde{b}}^{(k-1)} =
\bigoplus\nolimits_{j \in J}^{(k-1)}{Q_{b_j}^{(k-1)}}$ then by
Lemma~\ref{minimal elements suffice} 
$Q_{\tilde{b}}^{(k-1)} = \bigoplus\nolimits_{j \in
  J}^{(k)}{Q_{b_j}^{(k-1)}}$ and as $\pi_k(Q_{b_j}^{(k-1)}) \neq
\pi_k(Q_{0}^{(k-1)})$ and as $Q_b^{(k-1)}$ and thus
$Q_{\tilde{b}}^{(k-1)}$ is atomic
w.r.t.~$\mathop{\oplus}^{(k)}$ and $M$ this representation must be
trivial and we are done.

Therefore let us assume, that 
\begin{equation}
\label{norm_inequality_>}
\sum\limits_{i=1}^{t}{\sum\limits_{j \in J}{||(z_1^{i,j},z_2^{i,j})||_1}}
> \sum\limits_{i=1}^{t}{||(\bar{z}_1^i,\bar{z}_2^i)||_1}
\end{equation}
In the following, we construct a new representation
$Q_{\tilde{b}}^{(k-1)} = \bigoplus\nolimits_{j' \in
  J'}^{(k-1)}{Q_{b_j'}^{(k-1)}}$ and elements $(\tilde{z}_1^{i,j},\tilde{z}_2^{i,j},\tilde{z}_3^{i,j})$
whose corresponding sum \eqref{sum_elements} is smaller than the
minimally chosen sum. This contradiction proves that we have indeed
equality in \eqref{norm_inequality} and our claim is proved.

From \eqref{norm_inequality_>} and from \eqref{kleiner_gleich},
i.e., $z_1^{i,j} \sqsubseteq \bar{z}_1^i$ for
all $i$ and $j$, we conclude that there are indices $i_0,j_1,j_2$ such
that $z_2^{{i_0},{j_1}} \cdot z_2^{{i_0},{j_2}} < 0$. As $b_{j_1}, b_{j_2} \in
G^{\omega = 0}$ the sum $b_{j_1} + b_{j_2}$ was built during the
algorithm. We have $\omega_{k-1}(Q_{b_{j_1}+b_{j_2}}^{(k-1)}) =
\omega_{k-1}(Q_{b_{j_1}}^{(k-1)}) + \omega_{k-1}(Q_{b_{j_2}}^{(k-1)}) = 0$ and thus there is
no partially extended fiber with weight $\omega_{k-1}$ greater than
$0$ which reduces $Q_{b_{j_1}+b_{j_2}}^{(k-1)}$.
Consequently the partially extended
fiber $Q_{b_{j_1}+b_{j_2}}^{(k-1)}$ was either reduced to $Q_{0}^{(k-1)}$ by
sets $Q_{b_{j''}}^{(k-1)}$, $j'' \in J''$, during the else-clause of
the monoid-normal-form algorithm or the vector $b_{j_1} +b_{j_2}$ has
been added to the set $\bar{G}^{\omega =0}$. Then either $b_{j_1}
+b_{j_2} \in G^{\omega=0}$ or we find  sets $Q_{b_{j''}}^{(k-1)}$,
$j'' \in J''$, with $Q_{b_{j_1}+b_{j_2}}^{(k-1)}
=\bigoplus\nolimits_{j'' \in J''}^{(k)}{Q_{b_{j''}}^{(k-1)}}$ with
$b_{j''} \in G^{\omega = 0}$. In the former case, set $J'':=\{j''\}$
with $b_{j''} := b_{j_1}+b_{j_2}$.

This gives representations 
\begin{gather}
\notag
(z_1^{i,{j_1}},z_2^{i,{j_1}},z_3^{i,{j_1}}) + (z_1^{i,{j_2}},z_2^{i,{j_2}},z_3^{i,{j_2}}) =
\sum\limits_{j'' \in J''}{(\tilde{z}_1^{i,{j''}}, \tilde{z}_2^{i,{j''}},
  \tilde{z}_3^{i,{j''}})},\\
\notag
 (\tilde{z}_1^{i,{j''}}, \tilde{z}^{i,{j''}},
  \tilde{z}_3^{i,{j''}}) \in Q_{b_{j''}}^{(k-1)}, \; (\tilde{z}_1^{i,{j''}},
  \tilde{z}_2^{i,{j''}},\tilde{z}_3^{i,{j''}}) \sqsubseteq_{k} (z_1^{i,{j_1}},z_2^{i,{j_1}},z_3^{i,{j_1}}) +
  (z_1^{i,{j_2}},z_2^{i,{j_2}},z_3^{i,{j_2}})
\end{gather}
    for $i=1, \ldots, t$. As all $(\tilde{z}_1^{i,{j''}},
  \tilde{z}_2^{i,{j''}})$ lie in the same orthant as $(z_1^{i,{j_1}},z_2^{i,{j_1}}) +
  (z_1^{i,{j_2}},z_2^{i,{j_2}})$ we get:
\begin{equation*}
\biggl \lVert \sum\limits_{j''\in J''}{(\tilde{z}_1^{i,{j''}},
  \tilde{z}_2^{i,{j''}})}\biggr\rVert_1 = ||(z_1^{i,{j_1}},z_2^{i,{j_1}}) +
  (z_1^{i,{j_2}},z_2^{i,{j_2}})||_1 \leq ||(z_1^{i,{j_1}},z_2^{i,{j_1}})||_1 +
  ||(z_1^{i,{j_2}},z_2^{i,{j_2}})||_1
\end{equation*}
with strict inequality for $i=i_0$.
Thus, by replacing in $Q_{\tilde{b}}^{(k-1)} = \bigoplus\nolimits_{j \in
  J}^{(k-1)}{Q_{b_j}^{(k-1)}}$ the term $Q_{b_{j_1}}^{(k-1)} \mathop{\oplus}^{(k-1)}
Q_{b_{j_2}}^{(k-1)}$ by $\bigoplus\nolimits_{j'' \in
  J''}^{(k)}{Q_{b_{j''}}^{(k-1)}}$ we arrive at a new representation
$Q_{\tilde{b}}^{(k-1)} =\bigoplus\nolimits_{j' \in J'}^{(k-1)}{Q_{b_{j'}}^{(k-1)}}$
whose corresponding sum \eqref{sum_elements} is at most
\begin{equation*}
  \sum\limits_{i=1}^{t}{\sum\limits_{j' \in J'}{||(z_1^{i,{j'}},z_2^{i,{j'}})||_1}}
< \sum\limits_{i=1}^{t}{\sum\limits_{j \in J}{||(z_1^{i,j},z_2^{i,j})||_1}}
\end{equation*}
contradicting the minimality of the representation
$Q_{\tilde{b}}^{(k-1)} = \bigoplus\nolimits_{j \in 
  J}^{(k-1)}{Q_{b_j}^{(k-1)}}$. Therefore we have equality in
\eqref{norm_inequality} and thus $\tilde{b} \in G_{k-1}$ concluding
our proof. 
\end{proof}

\subsection{Intersecting with the appropriate orthant and testing reducibility}
\label{the sorting and reducing step}
In this subsection we want to illustrate the step of the
project-and-lift algorithm which follows the completion procedure in
each lifting step. This ``intersection and reducibility test'' is
illustrated by the dashed arrow in Figure \ref{intersection}.

\begin{figure}[ht]
\centering
\ifpdf
    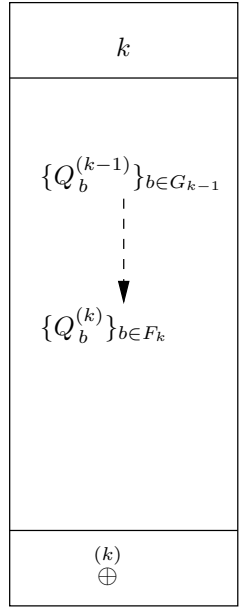
    \else
    \input{intersection.pstex_t}
    \fi
\caption{Intersecting with the appropriate orthant and dropping
  reducible partially extended fibers}
\label{intersection}
\end{figure}

\begin{algorithm}[ht]
\caption{Intersecting and testing reducibility}
\label{sort and reduce}
\begin{algorithmic}[1]
\REQUIRE{A set $G_{k-1} \subseteq M^{(k-1)}$ with the
properties: 
\begin{enumerate}
\item[(i)]
For every right-hand side $b \in M^{(k-1)}$ of a partially
extended fiber $Q_b^{(k-1)}$ which is atomic
w.r.t.~$\mathop{\oplus}^{(k)}$ and $M$ there exists $\tilde{b} \in
G_{k-1}$ with $\tilde{b} \preceq_{k} b$. 
\item[(ii)]
$b_i \npreceq_{k} b_j$ for $b_i, b_j \in
G_{k-1}$ with $b_i \neq b_j$
\end{enumerate}}
\ENSURE{A set $F_k \subseteq M^{(k)}$ of right-hand sides
%every element in $F_k$ is the right-hand side of an
%atomic partially extended fiber of order $k$
%w.r.t.~$\koplus{(k)}$. 
with:
\begin{enumerate}
\item[(i)]
For every right-hand side $b \in M^{(k)}$ of a partially
extended fiber $Q_b^{(k)}$ which is atomic
w.r.t.~$\mathop{\oplus}^{(k)}$ and $M$ there exists $\tilde{b} \in
F_k$ with $\tilde{b} \preceq_{k} b$. 
\item[(ii)]
$b_i \npreceq_{k} b_j$ for $b_i, b_j \in
F_{k}$ with $b_i \neq b_j$.
\end{enumerate} }
\STATE $F_{k} := \emptyset$
\FORALL {$b \in G_{k-1}$ with $Q_{b}^{(k)} \neq \emptyset$}
\IF {$Q_{b}^{(k)} \neq Q_{g}^{(k)} \mathop{\oplus}^{(k)}
  Q_{b-g}^{(k)}$ for all $b \neq g \in G_{k-1}$ with $b-g \in M$}
\STATE $F_k := F_k \cup \{b\}$
\ENDIF
\ENDFOR
\STATE \textbf{return } $F_k$
\end{algorithmic}
\end{algorithm}

\begin{Lemma}
\label{correctness of sort and reduce}
Algorithm \ref{sort and reduce} with input set $G_{k-1}$ terminates
and computes a set $F_k \subseteq M^{(k)}$ with the properties
(i) and (ii).
\end{Lemma}
\begin{proof}
Termination of Algorithm \ref{sort and reduce} is clear. We have
to show correctness of the algorithm. 
But this is easy as well: if $b \in M^{(k-1)}$ and $Q_{b}^{(k)} \neq
\emptyset$ then $b \in M^{(k)}$. Therefore $F_{k} \subseteq
M^{(k)}$. If $Q_{b}^{(k)}$ is atomic w.r.t.~$\mathop{\oplus}^{(k)}$ and $M$,
then $Q_{b}^{(k-1)}$ is atomic w.r.t.~$\mathop{\oplus}^{(k)}$ and $M$ as well,
because $Q_{b}^{(k)} \subseteq Q_{b}^{(k-1)}$ and every decomposition
of $Q_{b}^{(k-1)}$ would give a decomposition of $Q_{b}^{(k)}$.
This characteristic immediately implies property (i) of the output set
because we have property (i) of the input set.

To see property (ii) of the output set, suppose that there are
$b_1,b_2 \in F_k$ with $b_2 \preceq_{k} b_1$. Then, $Q_{b_1}^{(k)} =
Q_{b_2}^{(k)} \mathop{\oplus}^{(k)} Q_{b_1-b_2}^{(k)}$ and $b_1-b_2
\in \bar{S}^{(k)}$. In particular, $b_1-b_2 \in M$ which is a
contradiction as $b_1$ would not have been added to $F_k$ in this
case. This yields that $b_i \npreceq_{k} b_j$ for all $b_i,b_j \in
F_k$. Therefore Algorithm \ref{sort and reduce} is correct and
terminates. \end{proof}

\subsection{Refining the preorder}
\label{preprocessing the input set}

There is one more step to explain in the $k$-th lifting step of the
project-and-lift algorithm. This step is illustrated by the dotted
arrow in Figure \ref{refining}; it is implemented in
Algorithm~\ref{preprocessing algorithm}.

\begin{figure}[ht]
\centering

\ifpdf
    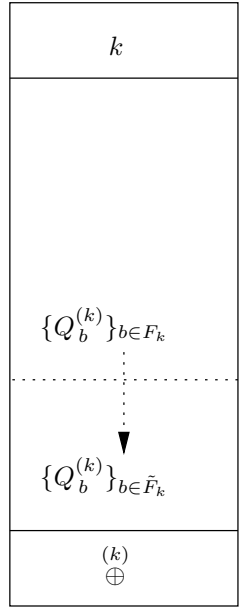
    \else
    \input{refining-partial-order.pstex_t}
    \fi

\caption{Refining the preorder to prepare the $k+1$-st lifting step}
\label{refining}
\end{figure}

\begin{algorithm}[ht]
\caption{Refining the preorder}
\label{preprocessing algorithm}
\begin{algorithmic}[1]
\REQUIRE{A set $F_k \subseteq M^{(k)}$ of right-hand sides
%every element in $F_k$ is the right-hand side of an
%atomic partially extended fiber of order $k$
%w.r.t.~$\koplus{(k)}$. 
with:
\begin{enumerate}
\item[(i)]
For every right-hand side $b \in M^{(k)}$ of a partially
extended fiber $Q_b^{(k)}$ which is atomic
w.r.t.~$\mathop{\oplus}^{(k)}$ and $M$ there exists $\tilde{b} \in
F_k$ with $\tilde{b} \preceq_{k} b$. 
\item[(ii)]
$b_i \npreceq_{k} b_j$ for $b_i, b_j \in
F_{k}$
\end{enumerate}}
\ENSURE{A set $\tilde{F}_{k} \subseteq M^{(k)}$ with the
following properties:
\begin{enumerate}
\item[(i)]
For every right-hand side $b \in M^{(k)}$ of a partially
extended fiber $Q_b^{(k)}$ which is atomic
w.r.t.~$\mathop{\oplus}^{(k)}$ and $M$ there exists $\tilde{b} \in
\tilde{F}_{k}$ with $\tilde{b} \preceq_{k+1} b$.
\item[(ii)]
$b_i \npreceq_{k+1} b_j$ for all $b_i, b_j \in \tilde{F}_{k}$ with
$b_i \neq b_j$.
\end{enumerate}}

%%\STATE $\tilde{F}_{k} := \emptyset$
\STATE Compute a set ${\mathcal L} = \{s^{(1)}, \ldots,
s^{(r)}\} \subseteq \bar{S}^{(k)}$ with:
\begin{displaymath}
\forall s \in \bar{S}^{(k)} \; \exists s^{(j)} \in
{\mathcal L} \text{ with } s^{(j)} \preceq_{k+1} s
\end{displaymath}
\STATE Set $\tilde{F}_k := \bigcup\limits_{b \in
  F_k}{\bigcup\limits_{s \in {\mathcal L}}\{b + s\}}$.
\FORALL{$b \in \tilde{F}_{k}$}
\IF{$\exists \bar{b} \in \tilde{F}_{k}$ with $\bar b \neq b$ and $\bar{b}
  \preceq_{k+1} b$}
\STATE $\tilde{F}_{k} := \tilde{F}_{k} \setminus \{b\}$
\ENDIF
\ENDFOR
\STATE \textbf{return } $\tilde{F}_k$
\end{algorithmic}
\end{algorithm}

\begin{Lemma}
\label{termination and correctness of preprocessing}
Algorithm \ref{preprocessing algorithm} terminates and is correct.
\end{Lemma}

\begin{proof} 
Termination of the above algorithm is clear once we have shown that
we can construct a finite set ${\mathcal L}$ with the property that 
\begin{displaymath}
\forall s\in \bar{S}^{(k)} \; \exists s^{(j)} \in
{\mathcal L} \text{ with } s^{(j)} \preceq_{k+1} s.
\end{displaymath}

For this, let us first construct generators for the monoid
$\bar{S}^{(k+1)}$. These can be found by considering the homogeneous
system of linear equations
\[
s=\sum_{j=k+2}^n \lambda_jA_j = \sum_{r=1}^t \alpha_r m_r,
\]
in the variables $s\in\Z^d$, $\alpha\in\Z^t_+$, and in
$\lambda\in\Z^{n-k-1}$. If we extract the values of $s$ for all
finitely many minimal homogeneous solutions of this linear system, we
obtain a generating set $\{\bar{s}_1,\ldots,\bar{s}_p\}$ for the
monoid $\bar{S}^{(k+1)}$.

Now let us consider the finite set
\begin{displaymath}
F=\left\{\sum_{i=1}^p
\lambda_i\bar{s}_i\colon 0\leq\lambda_i<1,i=1,\ldots,p\right\}\cap\{\lambda_{k+1}A_{k+1}+\ldots+\lambda_nA_n\colon
\lambda_i\in\Z\}.
\end{displaymath}
For each $f\in F$ we now consider the set
$\left(f+\bar{S}^{(k+1)}\right)\cap\bar{S}^{(k)}$ and construct a finite set
${\mathcal L}_f$ of vectors in
$\left(f+\bar{S}^{(k+1)}\right)\cap\bar{S}^{(k)}$ such that  
\begin{displaymath}
\forall s \in \left(f+\bar{S}^{(k+1)}\right)\cap\bar{S}^{(k)} \;
\exists s^{(j)} \in {\mathcal L}_f \text{ with } s^{(j)}
\preceq_{k+1} s. 
\end{displaymath}
Then ${\mathcal L}=\bigcup_{f\in F} {\mathcal L}_f
\subseteq\bar{S}^{(k)}$ is finite and has the desired property.

In order to construct  ${\mathcal L}_f$, let us consider the 
inhomogeneous system of linear equations and inequalities  
\[
f+\sum_{j=k+2}^n \lambda_jA_j = f+\sum_{r=1}^t \alpha_r m_r = s =
\sum_{j=k+1}^n \mu_jA_j = \sum_{r=1}^t \beta_r m_r
\]
in the variables $s\in\Z^d$, $\alpha,\beta\in\Z^t_+$, and in
$\lambda\in\Z^{n-k-1},\mu\in\Z^{n-k}$. The left-hand part states 
$s\in f+\bar{S}^{(k+1)}$ and the right-hand part encodes
$s\in\bar{S}^{(k)}$. Then a suitable set ${\mathcal L}_f$ can be found
by computing the finitely many minimal inhomogeneous solutions to this
linear system and by collecting the corresponding values of $s$. We
have thus proved that we may construct a finite set ${\mathcal L}$
admitting the claimed property. It remains to prove that the set
$\tilde{F}_k$ constructed from ${\mathcal L}$ admits the properties
claimed in Algorithm \ref{preprocessing algorithm}.

It is clear that $\tilde{F}_k \subseteq M^{(k)}$ because
$F_k \subseteq M^{(k)}$ and ${\mathcal L} \subseteq \bar{S}^{(k)}
\subseteq M$. Now let $b \in M^{(k)}$ be the right-hand side of an atomic
partially extended fiber w.r.t.~$\mathop{\oplus}^{(k)}$ and
$M$. Because of property (i) of the input-set, we find $\bar{b} \in
F_k$ with $\bar{b} \preceq_{k} b$, i.e., there is $\bar{s} \in
\bar{S}^{(k)}$ with $b = \bar{b} + \bar{s}$. As $\bar{s} \in
\bar{S}^{(k)}$ there is $s \in {\mathcal L}$ with
$s \preceq_{k+1} \bar{s}$ implying that there is
$\tilde{s} \in \bar{S}^{(k+1)}$ with $\bar{s} = s + \tilde{s}$. 
Let $\tilde{b} := \bar{b} + s$. Then clearly $\tilde{b} \preceq_{k+1}
b$. Thus either
$\tilde{b} \in \tilde{F}_{k}$  or
there is $b' \in \tilde{F}_{k}$ with $b' \preceq_{k+1}
\tilde{b} \preceq_{k+1} b$. This concludes the proof. \end{proof} 

\subsection{Initial input set and final output set}
The previous subsections have shown how one step of the
project-and-lift algorithm works. We have to perform $n$ of these
steps to obtain a set $G^*$ which contains the right-hand sides of all
atomic fibers w.r.t.~the monoid $M = \langle m_1, \ldots, m_t
\rangle$.
We start with the initial input set $F_0 = \{0\}$. This is a valid
input set because every non-empty extended fiber with right-hand side
$b \in M$, i.e., $b \in M^{(0)}$, also lies in $S^{(0)}$. 

\subsection{Simplifications for the lattice case}
\label{Section: Simplifications for the lattice case}
As already mentioned at the beginning of section~\ref{Section:
  Preliminaries of the project-and-lift algorithm}, the
project-and-lift algorithm to compute the atomic fibers of a matrix
with right-hand side $b$ varying on a lattice $\Lambda$ is much
easier to treat than the case of general monoids. 

The simplifications of the project-and-lift algorithm are based on
the fact that the difference of two arbitrary lattice vectors $b_1,b_2
\in \Lambda$, $b_1 -b_2$, is again a lattice vector. This fact has
implications for the preorder $\preceq_l$ on the right-hand side vectors
$b \in \Lambda$. Let $A \in \Z^{d \times n}$, $\Lambda \subseteq \Z^d$
a lattice and $0 \leq l \leq n$. Consider the preorder
introduced in section \ref{Section: Preliminaries of the
  project-and-lift algorithm}: $b_1 \preceq_l b_2$ if $b_2 -b_1 \in
\bar{S}^{(l)} = \{\lambda_{l+1}A_{l+1} + \ldots + \lambda_nA_n \, :
\, \lambda_i \in \Z\} \cap \Lambda$. As $\Lambda$ is a lattice we
obtain that $b_1-b_2 \in \bar{S}^{(l)}$ as well. This means: $b_1
\preceq_l b_2 \, \Leftrightarrow \, b_2 \preceq_l b_1$.  In other words, the
preorder is in fact an equivalence relation.

Our aim in this subsection is to simplify the refining step in our
project-and-lift algorithm.
Recall from section \ref{the completion procedure} that the input set
$\tilde{F}_{k-1}$ satisfies the following two properties which ensure
finiteness and correctness of the algorithm:
\begin{enumerate}
\item[(i)]
For every right-hand side $b \in M^{(k-1)}$ of a partially
extended fiber $Q_b^{(k-1)}$ which is atomic
w.r.t.~$\mathop{\oplus}^{(k-1)}$ and $M$ there exists $\tilde{b} \in
\tilde{F}_{k-1}$ with $\tilde{b} \preceq_{k} b$. 
\item[(ii)]
$b_i \npreceq_{k} b_j$ for $b_i \neq b_j \in
G_{k-1}$.
\end{enumerate}
In this subsection, we will define a new input set $\bar{F}_{k-1}$ of
Algorithm \ref{m-pl} which may be computed much easier than the set
$\tilde{F}_{k-1}$. Having defined this new input set we will expose
some properties of it. Finally we will show that the new input set
$\bar{F}_{k-1}$ is sufficient to guarantee finiteness and correctness
of the completion procedure, i.e., of Algorithm \ref{m-pl}.

Let $F_{k-1}$ be the output set of Algorithm \ref{sort and reduce} and
consider the following integer program:
\begin{equation}
\label{ip}
\begin{array}{r@{\; }l@{\,}c@{}rcl}
\min  & \lambda \\
\st  & \lambda_{k}A_{k} & + &\sum\limits_{i \geq k+1}{\lambda_i A_i}
& = &\sum\limits_{j=1}^{k}{\mu_j l_j}\\
& & & \lambda_{k} - \lambda &\leq &0\\
& & & - \lambda_{k} - \lambda &\leq& 0\\
& & & \lambda &\geq &1\\
& && \lambda, \lambda_i, \mu_j \in \Z 
\end{array}
\end{equation}

There are two possible cases: either the integer program \eqref{ip} is
infeasible or it admits an optimal solution $\lambda^*, \lambda_i^*,
\mu_j^*$. Consider the former case first and let $b_1,b_2 \in \Lambda$
with $b_1 \preceq_{k-1} b_2$. Then $b_1 \preceq_k b_2$. This is the
case because $b_2 -b_1 \in \Lambda$ and $b_2 - b_1 = \sum\nolimits_{i\geq
  k}{\lambda_iA_i}$, $\lambda_i \in \Z$. If the absolute value of
$\lambda_k$ was greater 
or equal than $1$ the difference $b_2-b_1$ would imply a feasible
solution of the integer program \eqref{ip}. Thus $\lambda_k = 0$ and
therefore we have $b_1 \preceq_k b_2$. In this case we set
$\bar{F}_{k-1}:= F_{k-1}$. 

Now consider the latter case. Let $s :=
\sum\nolimits_{i=k}^{n}{\lambda_i^*A_i}$. Again there are two
possible cases: either $s \preceq_k -s$ or $s \npreceq_k -s$. In the
former case we set $\bar{F}_{k-1}:= F_{k-1} \cup \{s\}$, in the latter
case we set $\bar{F}_{k-1} := F_{k-1} \cup \{\pm s\}$.

\begin{Lemma}
\label{weak first property}
Let \eqref{ip} admit an optimal solution $\lambda^*, \lambda_i^*,
\mu_j^*$. We assume w.l.o.g. that $\lambda_k^* \in \Z_+$.
Let $Q_b^{(k-1)}$ be atomic w.r.t.~$\mathop{\oplus}^{(k-1)}$ and
$\Lambda$. Then there is $\bar{b} \in F_{k-1}$ and $\lambda_b \in \Z$
with $\bar{b} + \lambda_bs \preceq_k b$.
\end{Lemma}

\begin{proof}
Let $b \in \Lambda$ with $Q_b^{(k-1)}$ atomic
w.r.t.~$\mathop{\oplus}^{(k-1)}$ and $\Lambda$. Then there is
$\bar{b} \in F_{k-1}$ with $\bar{b} \preceq_{k-1} b$, a consequence of
$F_{k-1}$ being the output set of Algorithm \ref{sort and reduce}. Let
$ \bar{s}:= 
b-\bar{b} = \sum\nolimits_{i\geq k}{\lambda_iA_i}$. Then there is $n
\in \Z$ and $0 \leq \bar{r} < \lambda_k^*$ with $\lambda_k = n \cdot
\lambda_k^* + \bar{r}$. We will show that $\bar{r} =0$. To this aim
consider $s_0 := \bar{s} -n \cdot s = \bar{r}A_k + \sum\nolimits_{i \geq
  k+1}{(\lambda_i - n \cdot \lambda_i^*)A_i}$. As $s, \bar{s} \in
  \Lambda$ we have $s_0 \in \Lambda$. If $\bar{r} \neq 0$ the
  feasible solution of \eqref{ip} implied by $s_0$ admits an
  objective value lower than $\lambda^*$, because $\bar{r} <
  \lambda_{k}^*$. This contradicts the optimality of the solution
  $\lambda^*, \lambda_i^*, \mu_j^*$. Therefore $\bar{r}= 0$ and thus
  $\bar{b} +n \cdot s \preceq_k b$.
\end{proof}

\begin{Lemma}
\label{second property}
For $b_i,b_j \in \bar{F}_{k-1}$  with $b_i \neq
b_j$ we have $b_i \npreceq_k b_j$.
\end{Lemma}
\begin{proof} 
There are a few different cases to consider: \smallbreak
\textit{Case 1:}  $b_i, b_j \in F_{k-1}$. 
Then $b_i \npreceq_k b_j$ as $F_{k-1}$ is the output set of Algorithm
\ref{sort and reduce}. \smallbreak
\textit{Case 2:} $b_i \in F_{k-1}$ and $b_j = s$.
We have to show that $s \npreceq_k b_j$. Suppose not and consider
$b_j-s = \sum\nolimits_{i \geq k+1}{\lambda_iA_i}$. But then $b_j =
\sum\nolimits_{i \geq k+1}{\lambda_iA_i} + \sum\nolimits_{j \geq
  k}{\lambda_j^*A_j}$ and thus $0 \preceq_{k-1} b_j$ which
contradicts $b_j$ being an element of $F_{k-1}$. \smallbreak
\textit{Case 3:} $b_i = s$ and $b_j = -s$. 
Here we have $s \npreceq_k -s$ by our assumptions.
\end{proof}

Lemma \ref{second property} implies that $\bar{F}_{k-1}$ defined as
above satisfies property $(ii)$ of the input set of Algorithm
\ref{m-pl}. We continue giving another property of the set
$\bar{F}_{k-1}$. 

\begin{Lemma}
\label{new representation}
Let $b \in \Lambda^{k-1}$. Then there is $\tilde{b} \preceq_k b$ with 
\begin{equation}
Q_{\tilde{b}}^{(k-1)} = \bigoplus\limits_{i}^{(k-1)}{Q_{b_i}^{(k-1)}}
\quad \text{ where } \quad b_i \in \bar{F}_{k-1}.
\end{equation}
\end{Lemma} 
\begin{proof}
As $b \in \Lambda^{(k-1)}$ there is $\bar{b} \preceq_{k-1} b$ with 
$$Q_{\bar{b}}^{(k-1)} = \bigoplus\limits_{j}^{(k-1)}{Q_{b_j}^{(k-1)}}
\quad \text{ where } b_j \in F_{k-1}. $$
This representation is a consequence of $F_{k-1}$ being the output
set of Algorithm \ref{sort and reduce}.
If the integer program \eqref{ip} is infeasible then $\bar{b} \preceq_k
b$ and our claim 
is proved. Therefore let \eqref{ip} admit an optimal
solution. W.l.o.g.~we assume that $\bar{F}_{k-1} = F_{k-1} \cup \{\pm
s\}$. Consider a decomposition of $Q_b^{(k-1)}$ into a restricted
Minkowski sum of partially extended fibers which are atomic
w.r.t.~this restricted Minkowski sum and the lattice $\Lambda$:
\begin{equation}
Q_b^{(k-1)} = \bigoplus\limits_{i}^{(k-1)}{Q_{b_i}^{(k-1)}}.
\end{equation}
With Lemma \ref{weak first property} we know that for each $b_i$
there is $\bar{b}_i \in F_{k-1}$ and $\lambda_{b_i} \in \Z$ such that 
$\bar{b}_i + \lambda_{b_i} \cdot s \preceq_k b_i$. We set
$\tilde{b}:= \sum\nolimits_{i}{\bar{b}_i +\lambda_{b_i} \cdot
  s}$. Then we have $\tilde{b} \preceq_k b$ and
$$Q_{\tilde{b}}^{(k-1)} =
\bigoplus\limits_{i}^{(k-1)}{(Q_{\bar{b}_i}^{(k-1)}
  \mathop{\oplus}^{(k-1)} \lambda_{b_i} Q_{s}^{(k-1)})} =
\bigoplus\limits_{i}^{(k-1)}{Q_{\bar{b_i}}^{(k-1)}}
\mathop{\oplus}^{(k-1)} (\sum\limits_{i}{\lambda_{b_i}})Q_s^{(k-1)}.$$
This proves our claim.
\end{proof}

\begin{algorithm}[ht]
\caption{Refining the preorder (equivalence relation) in the lattice case}
\label{refining equivalence classes (2)}
\begin{algorithmic}[1]
\REQUIRE{A lattice $\Lambda = \langle l_1, \ldots, l_t \rangle$ and a
  set $F_k \subseteq \Lambda^{(k)}$ of right-hand sides 
with:
\begin{enumerate}
\item[(i)]
For every right-hand side $b \in \Lambda^{(k)}$ of a partially
extended fiber $Q_b^{(k)}$ which is atomic
w.r.t.~$\mathop{\oplus}^{(k)}$ and $\Lambda$ there exists $\tilde{b} \in
F_k$ with $\tilde{b} \preceq_{k} b$. 
\item[(ii)]
$b_i \npreceq_{k} b_j$ for $b_i, b_j \in
F_{k}$ with $b_i \neq b_j$
\end{enumerate}}
\ENSURE{The set $\bar{F}_{k} \subseteq \Lambda^{(k)}$ defines as above.}
\STATE Solve the following integer program:
\begin{equation}
%\label{ip}
\begin{array}{r@{\; }l@{\,}c@{}rcl}
\min  & \lambda \\
\st  & \lambda_{k+1}A_{k+1} & + &\sum\limits_{i \geq k+2}{\lambda_i A_i}
& = &\sum\limits_{j=1}^{k}{\mu_j l_j}\\
& & & \lambda_{k+1} - \lambda &\leq &0\\
& & & - \lambda_{k+1} - \lambda &\leq& 0\\
& & & \lambda &\geq &1\\
& && \lambda, \lambda_i, \mu_j \in \Z 
\end{array}
\end{equation}

\IF{\eqref{ip} is feasible}
\STATE Let $\lambda^*,\lambda_i^*, \mu_j^*$ be an optimal solution of \eqref{ip}.
\STATE Set $s:=  \sum\nolimits_{i \geq
  k+1}{\lambda^*_iA_i}$.
\IF{$s \preceq_{k+1} -s$}
\STATE \textbf{return } $\bar{F}_{k} := F_k \cup \{s\}$
\ELSE
\STATE  \textbf{return } $\bar{F}_{k} := F_k \cup \{\pm s\}$
\ENDIF
\ELSE
\STATE \textbf{return } $\bar{F}_{k} := F_k$
\ENDIF
\end{algorithmic}
\end{algorithm}

Now we want to show that the input set $\bar{F}_{k-1}$ is sufficient
to guarantee finiteness and correctness of Algorithm \ref{m-pl}. An
input set admitting properties (i) and~(ii) is sufficient to do
so. We have seen in Lemma \ref{second property} that our set
$\bar{F}_{k-1}$ admits property~(ii). It admits property~(i) as
well if the integer program \eqref{ip} is infeasible. But it does not
admit this property in general if the integer program \eqref{ip} is
feasible. Note that in the proof of Algorithm \ref{m-pl} property~(i)
is only used to guarantee a representation
\eqref{representation of Q} according to Lemma \ref{Representation}
with the projection of the summands satisfying $\pi_{k}(Q_b^{(k-1)})
\neq \pi_{k}(Q_0^{(k-1)})$. But with Lemma \ref{new representation}
this representation may be guaranteed as well. Furthermore
$\pi_k(Q_{b}^{(k-1)}) \neq \pi_k(Q_0^{(k-1)})$. This is clear for $b
\in F_{k-1}$ because $0 \npreceq_{k-1} b$. As $0 \npreceq_k s$ we
have an analogue result for $\pi_k(Q_s^{(k-1)})$. This finally
implies that Algorithm \ref{m-pl} terminates and is correct when
given input set $\bar{F}_{k-1}$.

Besides the modification of the input set, Algorithm \ref{m-pl} stays the
same. Of course we may drop all tests if $b-g \in \Lambda$ during the
normal-form algorithm because for $b,g \in
\Lambda$ it is clear that the difference $b-g \in \Lambda$. The same
is valid for Algorithm \ref{sort and reduce}. It stays the same
except for the dropping of tests whether $b-g \in \Lambda$. 

Algorithm \ref{refining} is substituted by the
above Algorithm \ref{refining equivalence classes (2)}. It does not
compute the set $\tilde{F}_{k}$ but the set $\bar{F}_{k}$.

\begin{Lemma}
Algorithm \ref{refining equivalence classes (2)} terminates and is
correct. 
\end{Lemma}
\begin{proof}
This is a direct consequence of the discussion in this subsection.
\end{proof}

\clearpage

\section{First computational results}
\label{Section: First computational results}

We have created an implementation of the ``project-and-lift'' algorithm for
the lattice case (section~\ref{Section: Simplifications for the lattice case}).  The
implementation is written in Allegro Common Lisp~8.0 and C.  For the
computation of the minimal elements of partially extended fibers, we use the
library \texttt{libzsolve}, which is a part of 4ti2 \citep{4ti2}, version~1.3.1.  In this
section, we report on the computational experience with this code on several
test problems.  All computation times are given in CPU seconds on a 
Sun Fire V440 with UltraSPARC-IIIi processors
running at 1.6\,GHz.

\subsection{Results for number-partitioning problems}

We first consider the problem of partitioning a natural number $n$ into given
\emph{parts} (natural numbers) $a_1$, \dots, $a_k$ (with possible multiplicity).  To this
end, consider the set 
\begin{equation}
  P_n = \Bigl\{\, (x_1,\dots,x_k) \in \Z_+^k : n = \sum_{i=1}^k x_i \cdot a_i \,\Bigr\}.
\end{equation}
We are interested in a minimal set $\{n_1,\dots,n_q\}$ of natural numbers such
that the set $P_n$ of partitions of every number~$n$ is the Minkowski sum of
some of the sets $P_{n_j}$.   Thus we are interested in the atomic fibers
corresponding to the matrix 
\begin{equation}
  \left ( \begin{array}{c c c c c} 
      a_1 & a_2 & a_3 & \cdots & a_k 
    \end{array} \right ).
\end{equation}
We consider this problem for various sets of numbers $a_1$,\dots,$a_k$.
The results are shown in Table~\ref{table:number-part}.

\begin{table}[ht]
  \caption{Results for number-partitioning problems}
  \label{table:number-part}
  \begin{center}
  \def~{\hphantom0}
  \begin{tabular}{lcc}
    \toprule
    Parts & Atomic fibers & Time (s) \\
    \midrule
    1             & ~1 & ~~~~~1  \\
    1 2           & ~2 & ~~~~~1  \\
    1 2 3         & ~4 & ~~~~~1  \\
    1 2 3 4       & ~9 & ~~~~~1  \\
    1 2 3 4 5     & 32 & ~~~875  \\
    1 2 3 4 5 6   & 41 & $>$1000 \\
      2 3         & ~3 & ~~~~~1  \\
      2 3   5     & 14 & ~~~~~1  \\
      2 3   5   7 & 72 & 149661  \\
        3   5     & ~1 & ~~~~~1  \\
        3   5   7 & 30 & ~~~~~1  \\
        \bottomrule
  \end{tabular}
\end{center}
\end{table}

\subsection{Results for homogeneous number-partitioning problems}

Next we consider the problem of partitioning a natural number $n$ into given
natural numbers $a_1$, \dots, $a_k$ (with possible multiplicity), where we
prescribe the number of summands.  To this end, we consider the set
\begin{equation}
  P^m_n = \Bigl\{\, (x_1,\dots,x_k) \in \Z_+^k : n = \sum_{i=1}^k x_i \cdot a_i, \ 
  m = \sum_{i=1}^k x_i \,\Bigr\}.
\end{equation}
We are interested in a minimal set $\{(m_1, n_1),\dots,(m_q,n_q)\}$ of pairs 
$(m,n)$ such that the set $P^m_n$ of partitions of every number~$n$ into $m$
summands
is the Minkowski sum of some of the sets $P^{m_j}_{n_j}$.  Thus  we are interested in the atomic fibers
corresponding to the matrix  
\begin{equation}
\left ( \begin{array}{c c c c} 1 & 1 & \cdots & 1\\
    a_1 & a_2 & \cdots & a_k \end{array} \right ).
\end{equation}
Again we consider the problem for various sets of numbers $a_1$,\dots,$a_k$.
The results are shown in Table~\ref{table:number-part-homog}.  We remark that
the problem data $(1, 2, 3, 4)$ correspond to a problem equivalent to the one from
Example~\ref{Example: Atomic Fibers of Twisted Cubic}. 

\begin{table}[ht]
\caption{Results for homogeneous number-partitioning problems}
\label{table:number-part-homog}
{\def~{\hphantom0}
\begin{tabular}{lcc}
\toprule
Parts & Atomic fibers & Time (s)\\
\midrule
1                 & ~~~1 & ~~~~~1 \\    %hppi1
1 2               & ~~~2 & ~~~~~1 \\    %hppi2
1 2 3             & ~~~4 & ~~~~~1 \\    %hppi3
1 2 3 4           & ~~18 & ~~~~~1 \\    %hppi4 = x3
1 2 3 4 5         & ~~79 & ~~~~19 \\    % hppi5
1 2 3 5           & ~~12 & ~~~~~1 \\  % x4
1 2 3 6           & ~~35 & ~~~~~2 \\  % x5
1 2 3 7           & ~~19 & ~~~~~1 \\  % x6
1 2 3 8           & ~~58 & ~~~~30 \\  % x7
1 2 3 9           & ~~28 & ~~~~~2 \\  % x8
1 2 3 10          & ~~87 & ~~~206 \\ % x9
1 2 3 11          & ~~39 & ~~~~~6 \\ % x10
1 2 3 12          & ~122 & ~~1620 \\ % x11
1 2 3 13          & ~~52 & ~~~~21 \\ % x12
1 2 3 14          & ~163 & ~~5136 \\ % x13
1 2 3 15          & ~~67 & ~~~~72 \\  % x14
1 2 3 17          & ~~79 & ~~~216 \\ % x16
  2 3             & ~~~2 & ~~~~~1 \\ % homprime3
  2 3   5         & ~~~4 & ~~~~~1 \\ % homprime5
  2 3   5 7       & ~~26 & ~~~~~1 \\ % homprime7
  2 3   5 7    11 & ~262 & 152792 \\ % homprime11
\bottomrule
\end{tabular}}
\end{table}

\subsection{Results for Steinberger's sums of roots of unity}

One example that appears and was solved in \cite{Steinberger:04}
is the computation of the atomic fibers of the matrix
\[
\left(
\begin{array}{rrrrrrrrr}
1 & -1  &  0 & -1 &  1 &  0 &  0 &  0 & 0\\
0 &  1  & -1 &  0 & -1 &  1 &  0 &  0 & 0\\
0 &  0  &  0 &  1 & -1 &  0 & -1 &  1 & 0\\
0 &  0  &  0 &  0 &  1 & -1 &  0 & -1 & 1\\
\end{array}
\right).
\]
This matrix corresponds to a certain problem on $3\times 3$ tables
and has in fact $31$ atomic fibers and $79$ extended atomic
fibers.  The atomic fibers can be computed with our implementation in less
than one CPU second.

The next higher problem on $4\times 4$ tables leads to the matrix
\[
\left(
\begin{array}{rrrrrrrrrrrrrrrr}
1 & -1 &  0 &  0 & -1 &  1 &  0 &  0 &  0 &  0 &  0 &  0 &  0 &  0
& 0 & 0\\
0 &  1 & -1 &  0 &  0 & -1 &  1 &  0 &  0 &  0 &  0 &  0 &  0 &  0
& 0 & 0\\
0 &  0 &  1 & -1 &  0 &  0 & -1 &  1 &  0 &  0 &  0 &  0 &  0 &  0
& 0 & 0\\
0 &  0 &  0 &  0 &  1 & -1 &  0 &  0 & -1 &  1 &  0 &  0 &  0 &  0
& 0 & 0\\
0 &  0 &  0 &  0 &  0 &  1 & -1 &  0 &  0 & -1 &  1 &  0 &  0 &  0
& 0 & 0\\
0 &  0 &  0 &  0 &  0 &  0 &  1 & -1 &  0 &  0 & -1 &  1 &  0 &  0
& 0 & 0\\
0 &  0 &  0 &  0 &  0 &  0 &  0 &  0 &  1 & -1 &  0 &  0 & -1 &  1
& 0 & 0\\
0 &  0 &  0 &  0 &  0 &  0 &  0 &  0 &  0 &  1 & -1 &  0 &  0 & -1
& 1 & 0\\
0 &  0 &  0 &  0 &  0 &  0 &  0 &  0 &  0 &  0 &  1 & -1 &  0 &  0
& -1 & 1\\
\end{array}
\right).
\]
Our implementation was able to compute the~$12,675$ atomic fibers for this
matrix within 6.5 CPU days on a Sun Fire V440 with UltraSPARC-IIIi processors
running at 1.6\,GHz.

\clearpage
\bibliographystyle{plainnat}
\bibliography{iba-bib,francetelecom}
\end{document}

%% file: normaler-algorithmus.pdf_t
\begin{picture}(0,0)%
\includegraphics{normaler-algorithmus.pdf}%
\end{picture}%
\setlength{\unitlength}{4144sp}%
\begingroup\makeatletter\ifx\SetFigFont\undefined%
\gdef\SetFigFont#1#2#3#4#5{%
  \reset@font\fontsize{#1}{#2pt}%
  \fontfamily{#3}\fontseries{#4}\fontshape{#5}%
  \selectfont}%
\fi\endgroup%
\begin{picture}(3174,2724)(214,-2323)
\put(901,-2131){\makebox(0,0)[lb]{\smash{\SetFigFont{10}{12.0}{\familydefault}{\mddefault}{\updefault}{\color[rgb]{0,0,0}$+$}%
}}}
\put(2071,-16){\makebox(0,0)[lb]{\smash{\SetFigFont{10}{12.0}{\familydefault}{\mddefault}{\updefault}{\color[rgb]{0,0,0}$\{Q_{b}\}_{b \in G}$}%
}}}
\put(2476,-2131){\makebox(0,0)[lb]{\smash{\SetFigFont{10}{12.0}{\familydefault}{\mddefault}{\updefault}{\color[rgb]{0,0,0}$\oplus$}%
}}}
\put(766,-1636){\makebox(0,0)[lb]{\smash{\SetFigFont{10}{12.0}{\familydefault}{\mddefault}{\updefault}{\color[rgb]{0,0,0}$\{Q_{b}\}_{b \in F}$}%
}}}
\put(2071,-1636){\makebox(0,0)[lb]{\smash{\SetFigFont{10}{12.0}{\familydefault}{\mddefault}{\updefault}{\color[rgb]{0,0,0}$\{Q^{(k)}_{\,b}\}_{b \in G^*}$}%
}}}
\put(2526,-1114){\makebox(0,0)[lb]{\smash{\SetFigFont{10}{12.0}{\familydefault}{\mddefault}{\updefault}{\color[rgb]{0,0,0}reducible}%
}}}
\put(2521,-781){\makebox(0,0)[lb]{\smash{\SetFigFont{10}{12.0}{\familydefault}{\mddefault}{\updefault}{\color[rgb]{0,0,0}restriction,}%
}}}
\put(721,-961){\makebox(0,0)[lb]{\smash{\SetFigFont{10}{12.0}{\familydefault}{\mddefault}{\updefault}{\color[rgb]{0,0,0}completion}%
}}}
\put(2521,-961){\makebox(0,0)[lb]{\smash{\SetFigFont{10}{12.0}{\familydefault}{\mddefault}{\updefault}{\color[rgb]{0,0,0}dropping}%
}}}
\end{picture}%

%% file: normaler-algorithmus.pstex_t
\begin{picture}(0,0)%
\includegraphics{normaler-algorithmus.pstex}%
\end{picture}%
\setlength{\unitlength}{4144sp}%
\begingroup\makeatletter\ifx\SetFigFont\undefined%
\gdef\SetFigFont#1#2#3#4#5{%
  \reset@font\fontsize{#1}{#2pt}%
  \fontfamily{#3}\fontseries{#4}\fontshape{#5}%
  \selectfont}%
\fi\endgroup%
\begin{picture}(3174,2724)(214,-2323)
\put(901,-2131){\makebox(0,0)[lb]{\smash{\SetFigFont{10}{12.0}{\familydefault}{\mddefault}{\updefault}{\color[rgb]{0,0,0}$+$}%
}}}
\put(2071,-16){\makebox(0,0)[lb]{\smash{\SetFigFont{10}{12.0}{\familydefault}{\mddefault}{\updefault}{\color[rgb]{0,0,0}$\{Q_{b}\}_{b \in G}$}%
}}}
\put(2476,-2131){\makebox(0,0)[lb]{\smash{\SetFigFont{10}{12.0}{\familydefault}{\mddefault}{\updefault}{\color[rgb]{0,0,0}$\oplus$}%
}}}
\put(766,-1636){\makebox(0,0)[lb]{\smash{\SetFigFont{10}{12.0}{\familydefault}{\mddefault}{\updefault}{\color[rgb]{0,0,0}$\{Q_{b}\}_{b \in F}$}%
}}}
\put(2071,-1636){\makebox(0,0)[lb]{\smash{\SetFigFont{10}{12.0}{\familydefault}{\mddefault}{\updefault}{\color[rgb]{0,0,0}$\{Q^{(k)}_{\,b}\}_{b \in G^*}$}%
}}}
\put(2526,-1114){\makebox(0,0)[lb]{\smash{\SetFigFont{10}{12.0}{\familydefault}{\mddefault}{\updefault}{\color[rgb]{0,0,0}reducible}%
}}}
\put(2521,-781){\makebox(0,0)[lb]{\smash{\SetFigFont{10}{12.0}{\familydefault}{\mddefault}{\updefault}{\color[rgb]{0,0,0}restriction,}%
}}}
\put(721,-961){\makebox(0,0)[lb]{\smash{\SetFigFont{10}{12.0}{\familydefault}{\mddefault}{\updefault}{\color[rgb]{0,0,0}completion}%
}}}
\put(2521,-961){\makebox(0,0)[lb]{\smash{\SetFigFont{10}{12.0}{\familydefault}{\mddefault}{\updefault}{\color[rgb]{0,0,0}dropping}%
}}}
\end{picture}%

%% file: k-th_step.pdf_t
\begin{picture}(0,0)%
\includegraphics{k-th_step.pdf}%
\end{picture}%
\setlength{\unitlength}{4144sp}%
\begingroup\makeatletter\ifx\SetFigFont\undefined%
\gdef\SetFigFont#1#2#3#4#5{%
  \reset@font\fontsize{#1}{#2pt}%
  \fontfamily{#3}\fontseries{#4}\fontshape{#5}%
  \selectfont}%
\fi\endgroup%
\begin{picture}(2724,3624)(2689,-3223)
\put(3196, 74){\makebox(0,0)[lb]{\smash{\SetFigFont{10}{12.0}{\familydefault}{\mddefault}{\updefault}{\color[rgb]{0,0,0}$k-1$}%
}}}
\put(4681, 74){\makebox(0,0)[lb]{\smash{\SetFigFont{10}{12.0}{\familydefault}{\mddefault}{\updefault}{\color[rgb]{0,0,0}$k$}%
}}}
\put(3241,-3076){\makebox(0,0)[lb]{\smash{\SetFigFont{10}{12.0}{\familydefault}{\mddefault}{\updefault}{\color[rgb]{0,0,0}$\overset{(k-1)}{\oplus}$}%
}}}
\put(4546,-3076){\makebox(0,0)[lb]{\smash{\SetFigFont{10}{12.0}{\familydefault}{\mddefault}{\updefault}{\color[rgb]{0,0,0}$\overset{(k)}{\oplus}$}%
}}}
\put(2881,-1591){\makebox(0,0)[lb]{\smash{\SetFigFont{10}{12.0}{\familydefault}{\mddefault}{\updefault}{\color[rgb]{0,0,0}$\{Q_{\,b}^{(k-1)}\}_{b \in F_{k-1}}$ }%
}}}
\put(2881,-2491){\makebox(0,0)[lb]{\smash{\SetFigFont{10}{12.0}{\familydefault}{\mddefault}{\updefault}{\color[rgb]{0,0,0}$\{Q_{\,b}^{(k-1)}\}_{b \in \tilde{F}_{k-1}}$ }%
}}}
\put(4231,-691){\makebox(0,0)[lb]{\smash{\SetFigFont{10}{12.0}{\familydefault}{\mddefault}{\updefault}{\color[rgb]{0,0,0}$\{Q_{\,b}^{(k-1)}\}_{b \in G_{k-1}}$ }%
}}}
\put(4231,-1591){\makebox(0,0)[lb]{\smash{\SetFigFont{10}{12.0}{\familydefault}{\mddefault}{\updefault}{\color[rgb]{0,0,0}$\{Q_{\,b}^{(k)}\}_{b \in F_k}$ }%
}}}
\end{picture}%

%% file: k-th_step.pstex_t
\begin{picture}(0,0)%
\includegraphics{k-th_step.pstex}%
\end{picture}%
\setlength{\unitlength}{4144sp}%
\begingroup\makeatletter\ifx\SetFigFont\undefined%
\gdef\SetFigFont#1#2#3#4#5{%
  \reset@font\fontsize{#1}{#2pt}%
  \fontfamily{#3}\fontseries{#4}\fontshape{#5}%
  \selectfont}%
\fi\endgroup%
\begin{picture}(2724,3624)(2689,-3223)
\put(3196, 74){\makebox(0,0)[lb]{\smash{\SetFigFont{10}{12.0}{\familydefault}{\mddefault}{\updefault}{\color[rgb]{0,0,0}$k-1$}%
}}}
\put(4681, 74){\makebox(0,0)[lb]{\smash{\SetFigFont{10}{12.0}{\familydefault}{\mddefault}{\updefault}{\color[rgb]{0,0,0}$k$}%
}}}
\put(3241,-3076){\makebox(0,0)[lb]{\smash{\SetFigFont{10}{12.0}{\familydefault}{\mddefault}{\updefault}{\color[rgb]{0,0,0}$\overset{(k-1)}{\oplus}$}%
}}}
\put(4546,-3076){\makebox(0,0)[lb]{\smash{\SetFigFont{10}{12.0}{\familydefault}{\mddefault}{\updefault}{\color[rgb]{0,0,0}$\overset{(k)}{\oplus}$}%
}}}
\put(2881,-1591){\makebox(0,0)[lb]{\smash{\SetFigFont{10}{12.0}{\familydefault}{\mddefault}{\updefault}{\color[rgb]{0,0,0}$\{Q_{\,b}^{(k-1)}\}_{b \in F_{k-1}}$ }%
}}}
\put(2881,-2491){\makebox(0,0)[lb]{\smash{\SetFigFont{10}{12.0}{\familydefault}{\mddefault}{\updefault}{\color[rgb]{0,0,0}$\{Q_{\,b}^{(k-1)}\}_{b \in \tilde{F}_{k-1}}$ }%
}}}
\put(4231,-691){\makebox(0,0)[lb]{\smash{\SetFigFont{10}{12.0}{\familydefault}{\mddefault}{\updefault}{\color[rgb]{0,0,0}$\{Q_{\,b}^{(k-1)}\}_{b \in G_{k-1}}$ }%
}}}
\put(4231,-1591){\makebox(0,0)[lb]{\smash{\SetFigFont{10}{12.0}{\familydefault}{\mddefault}{\updefault}{\color[rgb]{0,0,0}$\{Q_{\,b}^{(k)}\}_{b \in F_k}$ }%
}}}
\end{picture}%

%% file: project-and-lift-scheme.pdf_t
\begin{picture}(0,0)%
\includegraphics{project-and-lift-scheme.pdf}%
\end{picture}%
\setlength{\unitlength}{4144sp}%
\begingroup\makeatletter\ifx\SetFigFont\undefined%
\gdef\SetFigFont#1#2#3#4#5{%
  \reset@font\fontsize{#1}{#2pt}%
  \fontfamily{#3}\fontseries{#4}\fontshape{#5}%
  \selectfont}%
\fi\endgroup%
\begin{picture}(6324,3624)(889,-3223)
\put(5311,-1591){\makebox(0,0)[lb]{\smash{\SetFigFont{17}{20.4}{\familydefault}{\mddefault}{\updefault}{\color[rgb]{0,0,0}$\mathbf{\cdots}$}%
}}}
\put(3511, 74){\makebox(0,0)[lb]{\smash{\SetFigFont{10}{12.0}{\familydefault}{\mddefault}{\updefault}{\color[rgb]{0,0,0}$2$}%
}}}
\put(4411, 74){\makebox(0,0)[lb]{\smash{\SetFigFont{10}{12.0}{\familydefault}{\mddefault}{\updefault}{\color[rgb]{0,0,0}$3$}%
}}}
\put(2611, 74){\makebox(0,0)[lb]{\smash{\SetFigFont{10}{12.0}{\familydefault}{\mddefault}{\updefault}{\color[rgb]{0,0,0}$1$}%
}}}
\put(1711, 74){\makebox(0,0)[lb]{\smash{\SetFigFont{10}{12.0}{\familydefault}{\mddefault}{\updefault}{\color[rgb]{0,0,0}$0$}%
}}}
\put(1171,164){\makebox(0,0)[lb]{\smash{\SetFigFont{10}{12.0}{\familydefault}{\mddefault}{\updefault}{\color[rgb]{0,0,0}$k$}%
}}}
\put(6391, 74){\makebox(0,0)[lb]{\smash{\SetFigFont{10}{12.0}{\familydefault}{\mddefault}{\updefault}{\color[rgb]{0,0,0}$n$}%
}}}
\put(6391,-3076){\makebox(0,0)[lb]{\smash{\SetFigFont{10}{12.0}{\familydefault}{\mddefault}{\updefault}{\color[rgb]{0,0,0}$\overset{(n)}{\oplus}$}%
}}}
\put(1666,-3076){\makebox(0,0)[lb]{\smash{\SetFigFont{10}{12.0}{\familydefault}{\mddefault}{\updefault}{\color[rgb]{0,0,0}$\overset{(0)}{\oplus}$}%
}}}
\put(2566,-3076){\makebox(0,0)[lb]{\smash{\SetFigFont{10}{12.0}{\familydefault}{\mddefault}{\updefault}{\color[rgb]{0,0,0}$\overset{(1)}{\oplus}$}%
}}}
\put(3466,-3076){\makebox(0,0)[lb]{\smash{\SetFigFont{10}{12.0}{\familydefault}{\mddefault}{\updefault}{\color[rgb]{0,0,0}$\overset{(2)}{\oplus}$}%
}}}
\put(4411,-3076){\makebox(0,0)[lb]{\smash{\SetFigFont{10}{12.0}{\familydefault}{\mddefault}{\updefault}{\color[rgb]{0,0,0}$\overset{(3)}{\oplus}$}%
}}}
\put(1036,-1141){\makebox(0,0)[lb]{\smash{\SetFigFont{10}{12.0}{\familydefault}{\mddefault}{\updefault}{\color[rgb]{0,0,0}$\preceq_k$}%
}}}
\put(946,-2491){\makebox(0,0)[lb]{\smash{\SetFigFont{10}{12.0}{\familydefault}{\mddefault}{\updefault}{\color[rgb]{0,0,0}$\preceq_{k+1}$}%
}}}
\put(2341,-691){\makebox(0,0)[lb]{\smash{\SetFigFont{10}{12.0}{\familydefault}{\mddefault}{\updefault}{\color[rgb]{0,0,0}$\{Q_{\,b}^{(0)}\}_{b \in G_0}$ }%
}}}
\put(2341,-1591){\makebox(0,0)[lb]{\smash{\SetFigFont{10}{12.0}{\familydefault}{\mddefault}{\updefault}{\color[rgb]{0,0,0}$\{Q_{\,b}^{(1)}\}_{b \in F_1}$ }%
}}}
\put(2341,-2491){\makebox(0,0)[lb]{\smash{\SetFigFont{10}{12.0}{\familydefault}{\mddefault}{\updefault}{\color[rgb]{0,0,0}$\{Q_{\,b}^{(1)}\}_{b \in \tilde{F}_1}$ }%
}}}
\put(3241,-2491){\makebox(0,0)[lb]{\smash{\SetFigFont{10}{12.0}{\familydefault}{\mddefault}{\updefault}{\color[rgb]{0,0,0}$\{Q_{\,b}^{(2)}\}_{b \in \tilde{F}_2}$ }%
}}}
\put(3241,-691){\makebox(0,0)[lb]{\smash{\SetFigFont{10}{12.0}{\familydefault}{\mddefault}{\updefault}{\color[rgb]{0,0,0}$\{Q_{\,b}^{(1)}\}_{b \in G_1}$ }%
}}}
\put(3241,-1591){\makebox(0,0)[lb]{\smash{\SetFigFont{10}{12.0}{\familydefault}{\mddefault}{\updefault}{\color[rgb]{0,0,0}$\{Q_{\,b}^{(2)}\}_{b \in F_2}$ }%
}}}
\put(1441,-2491){\makebox(0,0)[lb]{\smash{\SetFigFont{10}{12.0}{\familydefault}{\mddefault}{\updefault}{\color[rgb]{0,0,0}$\{Q_{\,b}^{(0)}\}_{b \in \tilde{F}_0}$ }%
}}}
\put(1441,-1591){\makebox(0,0)[lb]{\smash{\SetFigFont{10}{12.0}{\familydefault}{\mddefault}{\updefault}{\color[rgb]{0,0,0}$\{Q_{\,b}^{(0)}\}_{b \in F_0}$ }%
}}}
\put(4141,-691){\makebox(0,0)[lb]{\smash{\SetFigFont{10}{12.0}{\familydefault}{\mddefault}{\updefault}{\color[rgb]{0,0,0}$\{Q_{\,b}^{(2)}\}_{b \in G_2}$ }%
}}}
\put(4141,-2491){\makebox(0,0)[lb]{\smash{\SetFigFont{10}{12.0}{\familydefault}{\mddefault}{\updefault}{\color[rgb]{0,0,0}$\{Q_{\,b}^{(3)}\}_{b \in \tilde{F}_3}$ }%
}}}
\put(4141,-1591){\makebox(0,0)[lb]{\smash{\SetFigFont{10}{12.0}{\familydefault}{\mddefault}{\updefault}{\color[rgb]{0,0,0}$\{Q_{\,b}^{(3)}\}_{b \in F_3}$ }%
}}}
\put(5941,-691){\makebox(0,0)[lb]{\smash{\SetFigFont{10}{12.0}{\familydefault}{\mddefault}{\updefault}{\color[rgb]{0,0,0}$\{Q_{\,b}^{(n-1)}\}_{b \in G_{(n-1)}}$ }%
}}}
\put(6166,-1591){\makebox(0,0)[lb]{\smash{\SetFigFont{10}{12.0}{\familydefault}{\mddefault}{\updefault}{\color[rgb]{0,0,0}$\{Q_{\,b}^{(n)}\}_{b \in G^*}$ }%
}}}
\end{picture}%

%% file: project-and-lift-scheme.pstex_t
\begin{picture}(0,0)%
\includegraphics{project-and-lift-scheme.pstex}%
\end{picture}%
\setlength{\unitlength}{4144sp}%
\begingroup\makeatletter\ifx\SetFigFont\undefined%
\gdef\SetFigFont#1#2#3#4#5{%
  \reset@font\fontsize{#1}{#2pt}%
  \fontfamily{#3}\fontseries{#4}\fontshape{#5}%
  \selectfont}%
\fi\endgroup%
\begin{picture}(6324,3624)(889,-3223)
\put(5311,-1591){\makebox(0,0)[lb]{\smash{\SetFigFont{17}{20.4}{\familydefault}{\mddefault}{\updefault}{\color[rgb]{0,0,0}$\mathbf{\cdots}$}%
}}}
\put(3511, 74){\makebox(0,0)[lb]{\smash{\SetFigFont{10}{12.0}{\familydefault}{\mddefault}{\updefault}{\color[rgb]{0,0,0}$2$}%
}}}
\put(4411, 74){\makebox(0,0)[lb]{\smash{\SetFigFont{10}{12.0}{\familydefault}{\mddefault}{\updefault}{\color[rgb]{0,0,0}$3$}%
}}}
\put(2611, 74){\makebox(0,0)[lb]{\smash{\SetFigFont{10}{12.0}{\familydefault}{\mddefault}{\updefault}{\color[rgb]{0,0,0}$1$}%
}}}
\put(1711, 74){\makebox(0,0)[lb]{\smash{\SetFigFont{10}{12.0}{\familydefault}{\mddefault}{\updefault}{\color[rgb]{0,0,0}$0$}%
}}}
\put(1171,164){\makebox(0,0)[lb]{\smash{\SetFigFont{10}{12.0}{\familydefault}{\mddefault}{\updefault}{\color[rgb]{0,0,0}$k$}%
}}}
\put(6391, 74){\makebox(0,0)[lb]{\smash{\SetFigFont{10}{12.0}{\familydefault}{\mddefault}{\updefault}{\color[rgb]{0,0,0}$n$}%
}}}
\put(6391,-3076){\makebox(0,0)[lb]{\smash{\SetFigFont{10}{12.0}{\familydefault}{\mddefault}{\updefault}{\color[rgb]{0,0,0}$\overset{(n)}{\oplus}$}%
}}}
\put(1666,-3076){\makebox(0,0)[lb]{\smash{\SetFigFont{10}{12.0}{\familydefault}{\mddefault}{\updefault}{\color[rgb]{0,0,0}$\overset{(0)}{\oplus}$}%
}}}
\put(2566,-3076){\makebox(0,0)[lb]{\smash{\SetFigFont{10}{12.0}{\familydefault}{\mddefault}{\updefault}{\color[rgb]{0,0,0}$\overset{(1)}{\oplus}$}%
}}}
\put(3466,-3076){\makebox(0,0)[lb]{\smash{\SetFigFont{10}{12.0}{\familydefault}{\mddefault}{\updefault}{\color[rgb]{0,0,0}$\overset{(2)}{\oplus}$}%
}}}
\put(4411,-3076){\makebox(0,0)[lb]{\smash{\SetFigFont{10}{12.0}{\familydefault}{\mddefault}{\updefault}{\color[rgb]{0,0,0}$\overset{(3)}{\oplus}$}%
}}}
\put(1036,-1141){\makebox(0,0)[lb]{\smash{\SetFigFont{10}{12.0}{\familydefault}{\mddefault}{\updefault}{\color[rgb]{0,0,0}$\preceq_k$}%
}}}
\put(946,-2491){\makebox(0,0)[lb]{\smash{\SetFigFont{10}{12.0}{\familydefault}{\mddefault}{\updefault}{\color[rgb]{0,0,0}$\preceq_{k+1}$}%
}}}
\put(2341,-691){\makebox(0,0)[lb]{\smash{\SetFigFont{10}{12.0}{\familydefault}{\mddefault}{\updefault}{\color[rgb]{0,0,0}$\{Q_{\,b}^{(0)}\}_{b \in G_0}$ }%
}}}
\put(2341,-1591){\makebox(0,0)[lb]{\smash{\SetFigFont{10}{12.0}{\familydefault}{\mddefault}{\updefault}{\color[rgb]{0,0,0}$\{Q_{\,b}^{(1)}\}_{b \in F_1}$ }%
}}}
\put(2341,-2491){\makebox(0,0)[lb]{\smash{\SetFigFont{10}{12.0}{\familydefault}{\mddefault}{\updefault}{\color[rgb]{0,0,0}$\{Q_{\,b}^{(1)}\}_{b \in \tilde{F}_1}$ }%
}}}
\put(3241,-2491){\makebox(0,0)[lb]{\smash{\SetFigFont{10}{12.0}{\familydefault}{\mddefault}{\updefault}{\color[rgb]{0,0,0}$\{Q_{\,b}^{(2)}\}_{b \in \tilde{F}_2}$ }%
}}}
\put(3241,-691){\makebox(0,0)[lb]{\smash{\SetFigFont{10}{12.0}{\familydefault}{\mddefault}{\updefault}{\color[rgb]{0,0,0}$\{Q_{\,b}^{(1)}\}_{b \in G_1}$ }%
}}}
\put(3241,-1591){\makebox(0,0)[lb]{\smash{\SetFigFont{10}{12.0}{\familydefault}{\mddefault}{\updefault}{\color[rgb]{0,0,0}$\{Q_{\,b}^{(2)}\}_{b \in F_2}$ }%
}}}
\put(1441,-2491){\makebox(0,0)[lb]{\smash{\SetFigFont{10}{12.0}{\familydefault}{\mddefault}{\updefault}{\color[rgb]{0,0,0}$\{Q_{\,b}^{(0)}\}_{b \in \tilde{F}_0}$ }%
}}}
\put(1441,-1591){\makebox(0,0)[lb]{\smash{\SetFigFont{10}{12.0}{\familydefault}{\mddefault}{\updefault}{\color[rgb]{0,0,0}$\{Q_{\,b}^{(0)}\}_{b \in F_0}$ }%
}}}
\put(4141,-691){\makebox(0,0)[lb]{\smash{\SetFigFont{10}{12.0}{\familydefault}{\mddefault}{\updefault}{\color[rgb]{0,0,0}$\{Q_{\,b}^{(2)}\}_{b \in G_2}$ }%
}}}
\put(4141,-2491){\makebox(0,0)[lb]{\smash{\SetFigFont{10}{12.0}{\familydefault}{\mddefault}{\updefault}{\color[rgb]{0,0,0}$\{Q_{\,b}^{(3)}\}_{b \in \tilde{F}_3}$ }%
}}}
\put(4141,-1591){\makebox(0,0)[lb]{\smash{\SetFigFont{10}{12.0}{\familydefault}{\mddefault}{\updefault}{\color[rgb]{0,0,0}$\{Q_{\,b}^{(3)}\}_{b \in F_3}$ }%
}}}
\put(5941,-691){\makebox(0,0)[lb]{\smash{\SetFigFont{10}{12.0}{\familydefault}{\mddefault}{\updefault}{\color[rgb]{0,0,0}$\{Q_{\,b}^{(n-1)}\}_{b \in G_{(n-1)}}$ }%
}}}
\put(6166,-1591){\makebox(0,0)[lb]{\smash{\SetFigFont{10}{12.0}{\familydefault}{\mddefault}{\updefault}{\color[rgb]{0,0,0}$\{Q_{\,b}^{(n)}\}_{b \in G^*}$ }%
}}}
\end{picture}%

%% file: completion.pdf_t
\begin{picture}(0,0)%
\includegraphics{completion.pdf}%
\end{picture}%
\setlength{\unitlength}{4144sp}%
\begingroup\makeatletter\ifx\SetFigFont\undefined%
\gdef\SetFigFont#1#2#3#4#5{%
  \reset@font\fontsize{#1}{#2pt}%
  \fontfamily{#3}\fontseries{#4}\fontshape{#5}%
  \selectfont}%
\fi\endgroup%
\begin{picture}(2724,3624)(2689,-3223)
\put(3196, 74){\makebox(0,0)[lb]{\smash{\SetFigFont{10}{12.0}{\familydefault}{\mddefault}{\updefault}{\color[rgb]{0,0,0}$k-1$}%
}}}
\put(4681, 74){\makebox(0,0)[lb]{\smash{\SetFigFont{10}{12.0}{\familydefault}{\mddefault}{\updefault}{\color[rgb]{0,0,0}$k$}%
}}}
\put(3241,-3076){\makebox(0,0)[lb]{\smash{\SetFigFont{10}{12.0}{\familydefault}{\mddefault}{\updefault}{\color[rgb]{0,0,0}$\overset{(k-1)}{\oplus}$}%
}}}
\put(4546,-3076){\makebox(0,0)[lb]{\smash{\SetFigFont{10}{12.0}{\familydefault}{\mddefault}{\updefault}{\color[rgb]{0,0,0}$\overset{(k)}{\oplus}$}%
}}}
\put(2881,-2491){\makebox(0,0)[lb]{\smash{\SetFigFont{10}{12.0}{\familydefault}{\mddefault}{\updefault}{\color[rgb]{0,0,0}$\{Q_{\,b}^{(k-1)}\}_{b \in \tilde{F}_{k-1}}$ }%
}}}
\put(4231,-691){\makebox(0,0)[lb]{\smash{\SetFigFont{10}{12.0}{\familydefault}{\mddefault}{\updefault}{\color[rgb]{0,0,0}$\{Q_{\,b}^{(k-1)}\}_{b \in G_{k-1}}$ }%
}}}
\end{picture}%

%% file: completion.pstex_t
\begin{picture}(0,0)%
\includegraphics{completion.pstex}%
\end{picture}%
\setlength{\unitlength}{4144sp}%
\begingroup\makeatletter\ifx\SetFigFont\undefined%
\gdef\SetFigFont#1#2#3#4#5{%
  \reset@font\fontsize{#1}{#2pt}%
  \fontfamily{#3}\fontseries{#4}\fontshape{#5}%
  \selectfont}%
\fi\endgroup%
\begin{picture}(2724,3624)(2689,-3223)
\put(3196, 74){\makebox(0,0)[lb]{\smash{\SetFigFont{10}{12.0}{\familydefault}{\mddefault}{\updefault}{\color[rgb]{0,0,0}$k-1$}%
}}}
\put(4681, 74){\makebox(0,0)[lb]{\smash{\SetFigFont{10}{12.0}{\familydefault}{\mddefault}{\updefault}{\color[rgb]{0,0,0}$k$}%
}}}
\put(3241,-3076){\makebox(0,0)[lb]{\smash{\SetFigFont{10}{12.0}{\familydefault}{\mddefault}{\updefault}{\color[rgb]{0,0,0}$\overset{(k-1)}{\oplus}$}%
}}}
\put(4546,-3076){\makebox(0,0)[lb]{\smash{\SetFigFont{10}{12.0}{\familydefault}{\mddefault}{\updefault}{\color[rgb]{0,0,0}$\overset{(k)}{\oplus}$}%
}}}
\put(2881,-2491){\makebox(0,0)[lb]{\smash{\SetFigFont{10}{12.0}{\familydefault}{\mddefault}{\updefault}{\color[rgb]{0,0,0}$\{Q_{\,b}^{(k-1)}\}_{b \in \tilde{F}_{k-1}}$ }%
}}}
\put(4231,-691){\makebox(0,0)[lb]{\smash{\SetFigFont{10}{12.0}{\familydefault}{\mddefault}{\updefault}{\color[rgb]{0,0,0}$\{Q_{\,b}^{(k-1)}\}_{b \in G_{k-1}}$ }%
}}}
\end{picture}%

%% file: intersection.pdf_t
\begin{picture}(0,0)%
\includegraphics{intersection.pdf}%
\end{picture}%
\setlength{\unitlength}{4144sp}%
\begingroup\makeatletter\ifx\SetFigFont\undefined%
\gdef\SetFigFont#1#2#3#4#5{%
  \reset@font\fontsize{#1}{#2pt}%
  \fontfamily{#3}\fontseries{#4}\fontshape{#5}%
  \selectfont}%
\fi\endgroup%
\begin{picture}(1374,3624)(4039,-3223)
\put(4546,-3076){\makebox(0,0)[lb]{\smash{\SetFigFont{10}{12.0}{\familydefault}{\mddefault}{\updefault}{\color[rgb]{0,0,0}$\overset{(k)}{\oplus}$}%
}}}
\put(4681, 74){\makebox(0,0)[lb]{\smash{\SetFigFont{10}{12.0}{\familydefault}{\mddefault}{\updefault}{\color[rgb]{0,0,0}$k$}%
}}}
\put(4231,-691){\makebox(0,0)[lb]{\smash{\SetFigFont{10}{12.0}{\familydefault}{\mddefault}{\updefault}{\color[rgb]{0,0,0}$\{Q_{\,b}^{(k-1)}\}_{b \in G_{k-1}}$ }%
}}}
\put(4231,-1591){\makebox(0,0)[lb]{\smash{\SetFigFont{10}{12.0}{\familydefault}{\mddefault}{\updefault}{\color[rgb]{0,0,0}$\{Q_{\,b}^{(k)}\}_{b \in F_k}$ }%
}}}
\end{picture}%

%% file: intersection.pstex_t
\begin{picture}(0,0)%
\includegraphics{intersection.pstex}%
\end{picture}%
\setlength{\unitlength}{4144sp}%
\begingroup\makeatletter\ifx\SetFigFont\undefined%
\gdef\SetFigFont#1#2#3#4#5{%
  \reset@font\fontsize{#1}{#2pt}%
  \fontfamily{#3}\fontseries{#4}\fontshape{#5}%
  \selectfont}%
\fi\endgroup%
\begin{picture}(1374,3624)(4039,-3223)
\put(4546,-3076){\makebox(0,0)[lb]{\smash{\SetFigFont{10}{12.0}{\familydefault}{\mddefault}{\updefault}{\color[rgb]{0,0,0}$\overset{(k)}{\oplus}$}%
}}}
\put(4681, 74){\makebox(0,0)[lb]{\smash{\SetFigFont{10}{12.0}{\familydefault}{\mddefault}{\updefault}{\color[rgb]{0,0,0}$k$}%
}}}
\put(4231,-691){\makebox(0,0)[lb]{\smash{\SetFigFont{10}{12.0}{\familydefault}{\mddefault}{\updefault}{\color[rgb]{0,0,0}$\{Q_{\,b}^{(k-1)}\}_{b \in G_{k-1}}$ }%
}}}
\put(4231,-1591){\makebox(0,0)[lb]{\smash{\SetFigFont{10}{12.0}{\familydefault}{\mddefault}{\updefault}{\color[rgb]{0,0,0}$\{Q_{\,b}^{(k)}\}_{b \in F_k}$ }%
}}}
\end{picture}%

%% file: refining-partial-order.pdf_t
\begin{picture}(0,0)%
\includegraphics{refining-partial-order.pdf}%
\end{picture}%
\setlength{\unitlength}{4144sp}%
\begingroup\makeatletter\ifx\SetFigFont\undefined%
\gdef\SetFigFont#1#2#3#4#5{%
  \reset@font\fontsize{#1}{#2pt}%
  \fontfamily{#3}\fontseries{#4}\fontshape{#5}%
  \selectfont}%
\fi\endgroup%
\begin{picture}(1374,3624)(2689,-3223)
\put(3241,-3076){\makebox(0,0)[lb]{\smash{\SetFigFont{10}{12.0}{\familydefault}{\mddefault}{\updefault}{\color[rgb]{0,0,0}$\overset{(k)}{\oplus}$}%
}}}
\put(3286, 74){\makebox(0,0)[lb]{\smash{\SetFigFont{10}{12.0}{\familydefault}{\mddefault}{\updefault}{\color[rgb]{0,0,0}$k$}%
}}}
\put(2881,-2491){\makebox(0,0)[lb]{\smash{\SetFigFont{10}{12.0}{\familydefault}{\mddefault}{\updefault}{\color[rgb]{0,0,0}$\{Q_{\,b}^{(k)}\}_{b \in \tilde{F}_{k}}$}%
}}}
\put(2881,-1591){\makebox(0,0)[lb]{\smash{\SetFigFont{10}{12.0}{\familydefault}{\mddefault}{\updefault}{\color[rgb]{0,0,0}$\{Q_{\,b}^{(k)}\}_{b \in F_{k}}$ }%
}}}
\end{picture}%

%% file: refining-partial-order.pstex_t
\begin{picture}(0,0)%
\includegraphics{refining-partial-order.pstex}%
\end{picture}%
\setlength{\unitlength}{4144sp}%
\begingroup\makeatletter\ifx\SetFigFont\undefined%
\gdef\SetFigFont#1#2#3#4#5{%
  \reset@font\fontsize{#1}{#2pt}%
  \fontfamily{#3}\fontseries{#4}\fontshape{#5}%
  \selectfont}%
\fi\endgroup%
\begin{picture}(1374,3624)(2689,-3223)
\put(3241,-3076){\makebox(0,0)[lb]{\smash{\SetFigFont{10}{12.0}{\familydefault}{\mddefault}{\updefault}{\color[rgb]{0,0,0}$\overset{(k)}{\oplus}$}%
}}}
\put(3286, 74){\makebox(0,0)[lb]{\smash{\SetFigFont{10}{12.0}{\familydefault}{\mddefault}{\updefault}{\color[rgb]{0,0,0}$k$}%
}}}
\put(2881,-2491){\makebox(0,0)[lb]{\smash{\SetFigFont{10}{12.0}{\familydefault}{\mddefault}{\updefault}{\color[rgb]{0,0,0}$\{Q_{\,b}^{(k)}\}_{b \in \tilde{F}_{k}}$}%
}}}
\put(2881,-1591){\makebox(0,0)[lb]{\smash{\SetFigFont{10}{12.0}{\familydefault}{\mddefault}{\updefault}{\color[rgb]{0,0,0}$\{Q_{\,b}^{(k)}\}_{b \in F_{k}}$ }%
}}}
\end{picture}%